\documentclass[10pt,a4paper]{article}

\RequirePackage[colorlinks,hyperindex]{hyperref}

\usepackage{graphicx,subfigure}

\usepackage[square,numbers]{natbib}

\usepackage{pstricks,pst-plot}

\usepackage{amsmath}
\usepackage{amsthm}
\usepackage{amsfonts}
\usepackage{amssymb}
\usepackage{latexsym}
\usepackage{stackrel}
\usepackage[T1]{fontenc}
\usepackage{fourier}

\usepackage{lastpage}

\let\mathnumsetfont\mathbb
\newcommand\Rset{\mathnumsetfont R} 

\newtheorem{teo*}{Theorem}
\newtheorem{proposition}{Proposition}

\def\barF{\overline{F}}

\begin{document}

\title{A New Family of Asymmetric Distributions for Modeling Light-Tailed and Right-Skewed Data}
\author{Meitner Cadena \\
Departamento de Ciencias Exactas, Universidad de las Fuerzas Armadas - ESPE \\
Sangolqu\'{\i}, Ecuador
}
\maketitle

\begin{abstract}
A new three-parameter cumulative distribution function defined on $(\alpha,\infty)$, for some $\alpha\geq0$, with asymmetric probability density function and showing exponential decays at its both tails, is introduced.
The new distribution is near to familiar distributions like the gamma and log-normal distributions, but this new one shows own elements and thus does not generalize neither of these distributions.
Hence, the new distribution constitutes a new alternative to fit values showing light-tailed behaviors.
Further, this new distribution shows great flexibility to fit the bulk of data by tuning some parameters.
We refer to this new distribution as the generalized exponential log-squared distribution (GEL-S).
Statistical properties of the GEL-S distribution are discussed.
The maximum likelihood method is proposed for estimating the model parameters, but incorporating adaptations in computational procedures due to difficulties in the manipulation of the parameters.
The performance of the new distribution is studied using simulations.
Applications to real data sets coming from different domains are showed.

\begin{center}
\end{center}
\end{abstract}


\section{Introduction} 


In a number of domains such as medical applications, atmospheric sciences, microbiology, environmental science, and reliability theory among others, data are positive, right-skewed, with their highest values decaying exponentially.
Among the most suitable models used by researchers and practitioners to deal with this kind of data are usually parametric distributions as the log-normal, gamma and Weibull distributions.
However, known distributions are not always enough to reach a good fit to the data.
This has motivated the interest in the development of more flexible and better adapted distributions, 
which have been generated using different strategies such as the combination of known distributions \cite{RubioHong2016}, introduction of new parameters in given distributions \cite{MarshallOlkin1997}, transformation of known distributions \cite{GuptaGuptaGupta1998}, and junction of two distributions by splicing \cite{RubioSteel2014}.

In this paper, a new procedure to develop new distributions is proposed.
The aim is to guarantee that a probability density function (pdf) $f(x)$ defined for $x>\alpha$, for some $\alpha\in\Rset$, exponentially decays to 0 as $x\to\alpha^+$ and as $x\to\infty$.
An advantage of this condition is that this itself still holds if any polynomial $x^\beta$ with $\beta\in\Rset$ is included as a factor in such pdf.
In this way, the new distribution would have great flexibility in neighborhoods of 0 and $\infty$ by controlling $\beta$,
thus capturing a wide variety of shapes and tail behaviors.
We refer to this new distribution as the generalized exponential log-squared distribution (GEL-S).

The above-mentioned features for pdfs are also satisfied by the log-normal and related distributions.
Further, the log-normal distribution may be considered a particular case of GEL-S,
however, as will be seen later, the latter does not generalize the log-normal distribution.

The aim of this paper is two-fold.
First, to study statistical properties of the GEL-S distribution and methods for estimating its parameters.
Second, to provide empirical evidence on the great flexibility of the GEL-S distribution to fit real light-tailed and right-skewed data from different domains.
For numerical assessments, the implementation of this model is done using functions in the R software \cite{RCoreTeam}.

In the next section, the pdf associated to the new three-parameter distribution is introduced by considering the condition on pdfs indicated above, and explicit expressions of its cumulative distribution function (cdf) and survival function (sf) are provided in some cases.
Further, closeness of the new distribution with well-known distributions is discussed.
Section \ref{sec2} presents statistical properties of the new distribution.
Section \ref{sec3} is devoted to the maximum likelihood method for estimating the parameters of the new distribution.
In Section \ref{sec4}, the performance of the parameter estimation method is studied using simulations.
Section \ref{sec5} shows applications of the new distributions to real data sets coming from different domains.
Section \ref{sec6} concludes the paper presenting discussions and conclusions and next further steps.
Proofs are presented in the annexe.

\section{The Generalized Exponential Log-Squared Distribution}

In this section, the GEL-S distribution is introduced.
The pdf of the new cdf is defined by
$$
f(x):=C\,x^\beta e^{-(2\gamma^2)^{-1}\left(\log(x-\alpha)\right)^2},\quad x>\alpha,\ \textrm{with $\alpha\geq0$, $\beta\in\Rset$ and $\gamma>0$},
$$
where $C$ is the normalizing constant.

This function holds exponential decays at its tails:
writing
$$
x^\beta e^{-(2\gamma^2)^{-1}\left(\log(x-\alpha)\right)^2}
=e^{-\left(\log(x-\alpha)\right)^2\left((2\gamma^2)^{-1}-\beta\,\frac{\log x}{\left(\log(x-\alpha)\right)^2}\right)}
$$
and noting that, if $\alpha=0$,
$$
\lim_{x\to\alpha^+}\frac{\log x}{\left(\log(x-\alpha)\right)^2}
=\lim_{x\to0^+}\frac{1}{\log x}
=0,
$$
or if $\alpha>0$,
$$
\lim_{x\to\alpha^+}\frac{\log x}{\left(\log(x-\alpha)\right)^2}
=\log\alpha\times\lim_{x\to\alpha^+}\frac{1}{\left(\log(x-\alpha)\right)^2}=0,
$$
and, by applying the L'H\^{o}pital rule,
$$
\lim_{x\to\infty}\frac{\log x}{\left(\log(x-\alpha)\right)^2}
=\frac{1}{2}\lim_{x\to\infty}\frac{1}{\log(x-\alpha)}
=0,
$$
then we have, for any $\beta\in\Rset$,
$$
\lim_{x\to\alpha^+}f(x)=0,\quad
\lim_{x\to\infty}f(x)=0.
$$
Further, this means that both tails of this function are light \cite{BinghamGoldieTeugels,resnick}, which implies that $f$ reaches 0 very fast when $x\to\alpha^+$ or $x\to\infty$.

Due to difficulties in the manipulation of $f$ for any $\beta\in\Rset$, for instance for computing integrals of this function, 
this study was limited to cases when $\beta$ takes non-negative integer values.
Therefore, in this paper, the definition of the pdf will be considered as
\begin{equation}\label{deff}
f(x):=C\,x^k e^{-(2\gamma^2)^{-1}\left(\log(x-\alpha)\right)^2},\quad x>\alpha,\ \textrm{with $\alpha\geq0$, $k=0,1,2,\ldots,$ and $\gamma>0$},
\end{equation}
where $C=\left(\gamma\sqrt{2\pi}\,\sum_{i=0}^k{{k}\choose{i}}\alpha^{k-i}e^{(i+1)^2\gamma^2/2}\right)^{-1}$ is the normalizing constant.
The deduction of $C$ is presented in Annexe \ref{Proofs}.

The cdf is then, for $x>\alpha$,
\begin{eqnarray}
F(x) & := & \int_\alpha^x f(z)\,dz \nonumber \\
 & = & \gamma C\sqrt{2\pi}\,\sum_{i=0}^k{{k}\choose{i}}\alpha^{k-i}e^{(i+1)^2\gamma^2/2}\Phi\left(\frac{1}{\gamma}\left(\log(x-\alpha)-(i+1)\gamma^2\right)\right),
\label{defF}
\end{eqnarray}
where $\Phi$ is the cdf of a standard normal random variable (rv).
The deduction of $F$ is presented in Annexe \ref{Proofs}.
From (\ref{defF}) the sf $\barF$ associated to $F$ may be deduced by using its definition $\barF:=1-F$, but following similar computations as in the deduction of $F$ and using the property $1-\Phi(x)=\Phi(-x)$, the following expression for $\barF$ is obtained, for $x>\alpha$:
\begin{eqnarray*}
\lefteqn{\barF(x)=\int_x^\infty f(z)\,dz} \\
 & & =\gamma C\sqrt{2\pi}\,\sum_{i=0}^k{{k}\choose{i}}\alpha^{k-i}e^{(i+1)^2\gamma^2/2}\Phi\left(-\frac{1}{\gamma}\left(\log(x-\alpha)-(i+1)\gamma^2\right)\right).
\end{eqnarray*}

Relating $f$ with the pdf of a log-normal distribution with parameters $\mu$ and $\sigma^2$, writing
$$
x^k\,e^{-(2\gamma^2)^{-1}\left(\log x\right)^2}
=e^{2^{-1}\gamma^2(k+1)^2}\,x^{-1}\,e^{-(2\gamma^2)^{-1}\left(\log x-\gamma^2(k+1)\right)^2},
$$
gives that the former distribution becomes the latter one if $\alpha=0$, $\gamma=\sigma$, and $\gamma=\sqrt{\mu\big/(k+1)}$.
Hence, the log-normal distribution is a particular case of the GEL-S distribution, implying that the GEL-S distribution might thus inherit the importance that the log-normal distribution has taken to model data \cite{CrowShimizu,LimpertStahelAbbt2001}.
However, the new distribution is not an extension of the log-normal distribution since the latter is built when considering the rv $\log X$ where $X$ is a rv following a normal distribution, but the introduction of $x=e^y$ in $F(x)$ gives an expression that is not related to any expression based on normal rvs.
The reader is referred to \cite{Yuan1933,CohenWhitten1980,GuptaLvin2005,SinghSharmaRathiSingh2012} for further details on the log-normal distribution and its generalizations.

As discussed above, the GEL-S distribution is close to the log-normal distribution.
Other pdfs that are close to the new pdf in terms of their structures are presented in Table \ref{tab01}, where the new pdf is included in order to appreciate similarities and differences among them.
Through these pdfs two main functions multiplying each other are identified: one that is based on the exponential function and a second function formed by the remaining part.
According to the functions based on the exponential function, the GEL-S, two-parameter log-normal and three-parameter log-normal distributions are very similar to each other, whereas on the second functions the GEL-S and gamma distributions are close to each other.

\begin{table}[!ht]
\centering
\begin{tabular}{lccc}
\hline
\multicolumn{1}{c}{Distribution} & Parameters & Support & pdf \\
\hline
GEL-S & $\alpha\geq0$, $k=0,1,2,\ldots$, $\gamma>0$ & $x>\alpha$ & $C\,x^k\, e^{-\frac{\left(\log(x-\alpha)\right)^2}{2\gamma^2}}$ \\
Two-parameter log-normal & $\mu\in\Rset$, $\sigma>0$ & $x>0$ & $\frac{x^{-1}}{\sigma\sqrt{2\pi}}e^{-\frac{(\log x-\mu)^2}{2\sigma^2}}$ \\
Three-parameter log-normal & $\delta,\mu\in\Rset$, $\sigma>0$ & $x>\delta$ & $\frac{(x-\delta)^{-1}}{\sigma\sqrt{2\pi}}e^{-\frac{(\log(x-\delta)-\mu)^2}{2\sigma^2}}$ \\
Gamma & $\alpha,\beta>0$ & $x>0$ & $\frac{\beta^\alpha}{\Gamma(\alpha)}x^{\alpha-1}e^{-\beta x}$ \\
\hline
\end{tabular}
\caption{Distributions that are close to the GEL-S distribution}
\label{tab01}
\end{table}

Plots of pdfs and cdfs of the GEL-S, two-parameter log-normal and three-parameter log-normal distributions are exhibited in Fig. \ref{fig01}.
Left plots concern pdfs and right plots their corresponding cdfs.
In top plots, the GEL-S and two-parameter log-normal distributions are compared by varying their parameters.
Note that the supports of the positive parts of the pdfs and cdfs are not the same for both distributions: the one of the log-normal distribution begins at $x=0^+$ and is in general slightly wider than that of the GEL-S distribution that begins at $x=\alpha^+>0$.
In these plots, the cdf and pdf of the log-normal distributions are surrounded by the ones of the GEL-S distributions, reflecting the fact that the two-parameter log-normal distribution is a particular case of the GEL-S distribution as discussed above.
This enclosure is done by varying $\gamma$ of the GEL-S distribution.
In bottom plots, the GEL-S and three-parameter log-normal distributions are compared as in the previous comparisons, but considering the same support for both distributions by taking $\alpha=\delta$.
Now, the cdf and pdf of the log-normal distribution are partially surrounded by the ones of the GEL-S distribution, namely at the right side of the curves.

\begin{figure}[!ht]
\centering \includegraphics[scale=0.30]{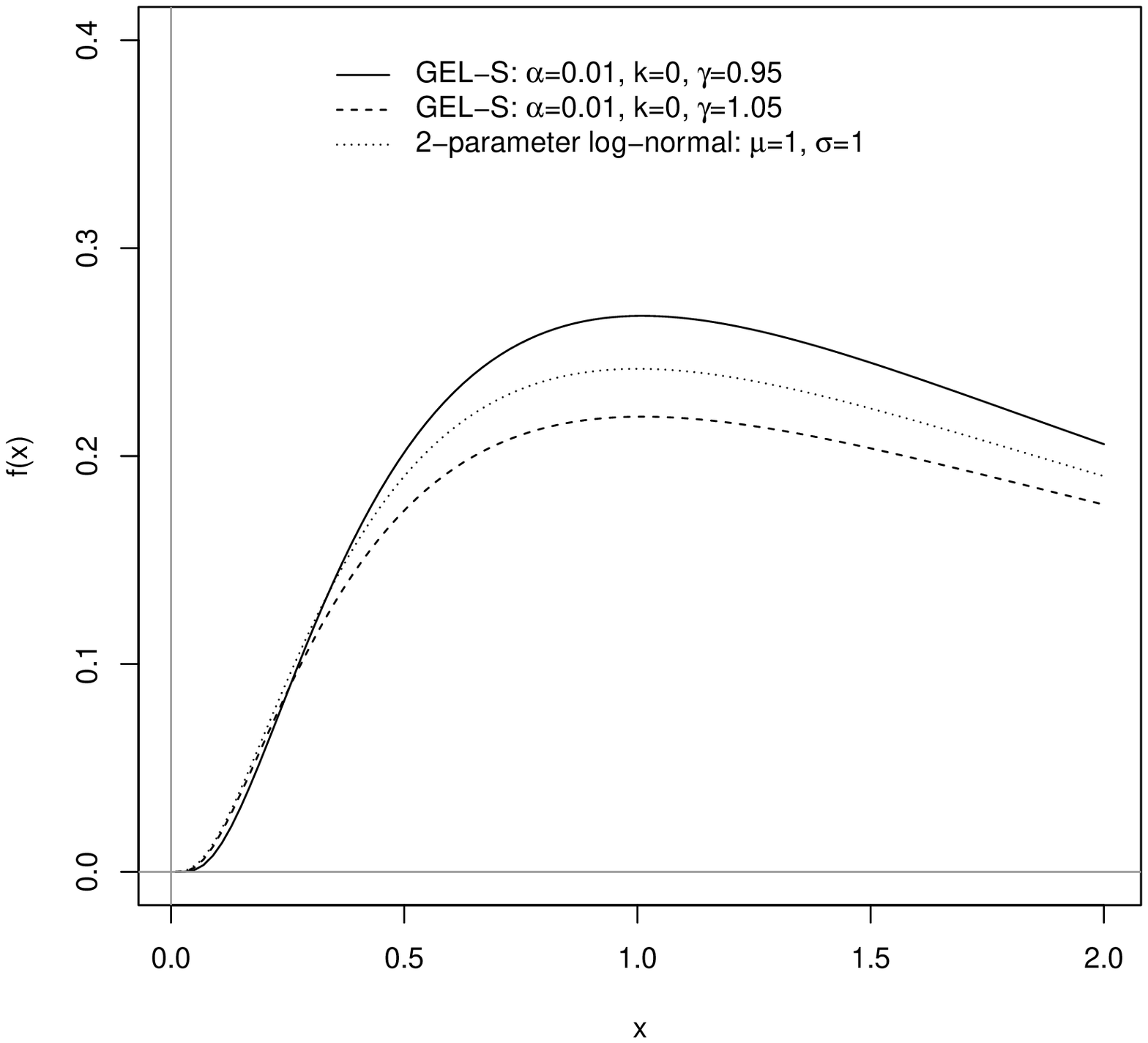} 
\centering \includegraphics[scale=0.30]{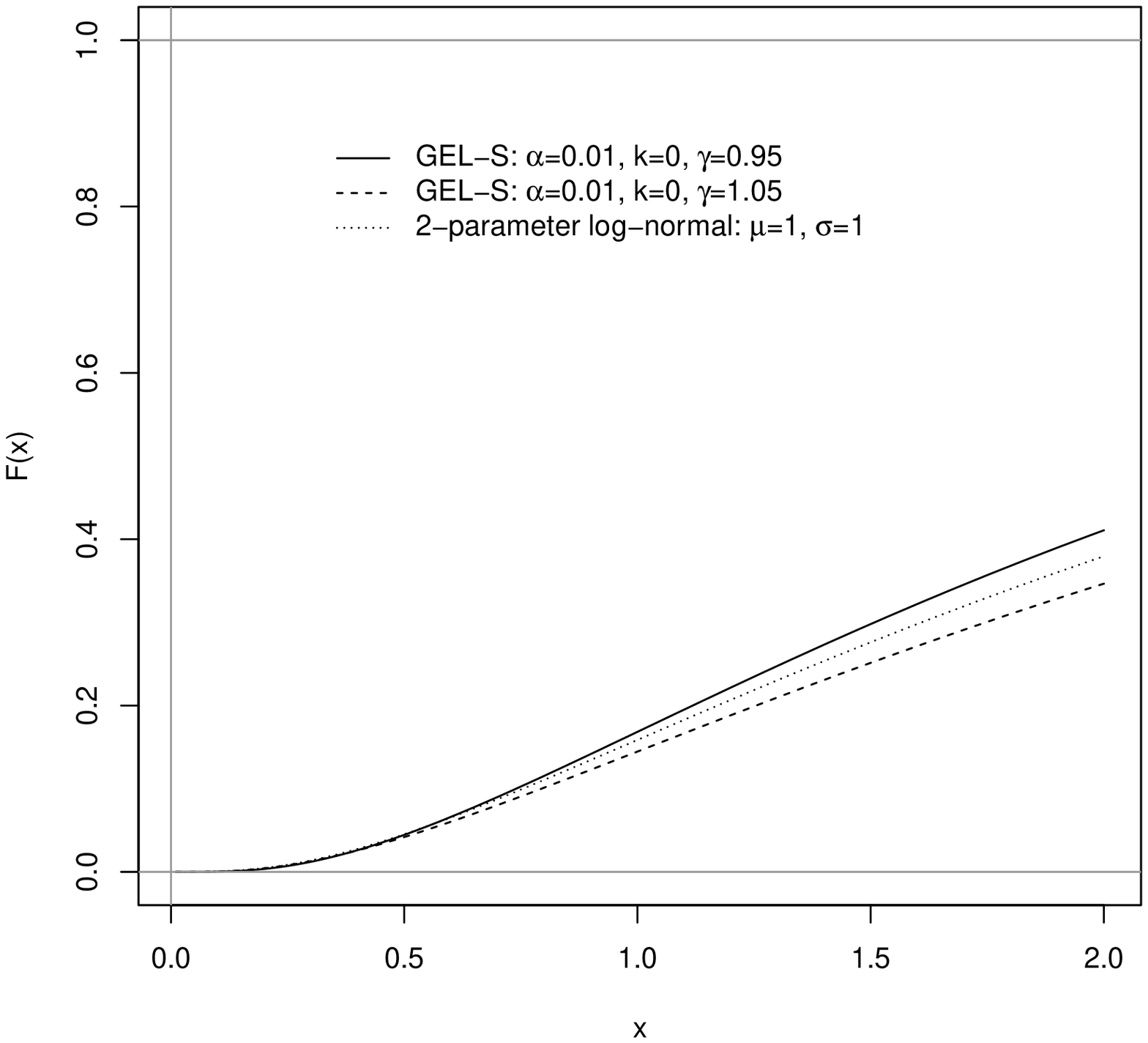} 
\centering \includegraphics[scale=0.30]{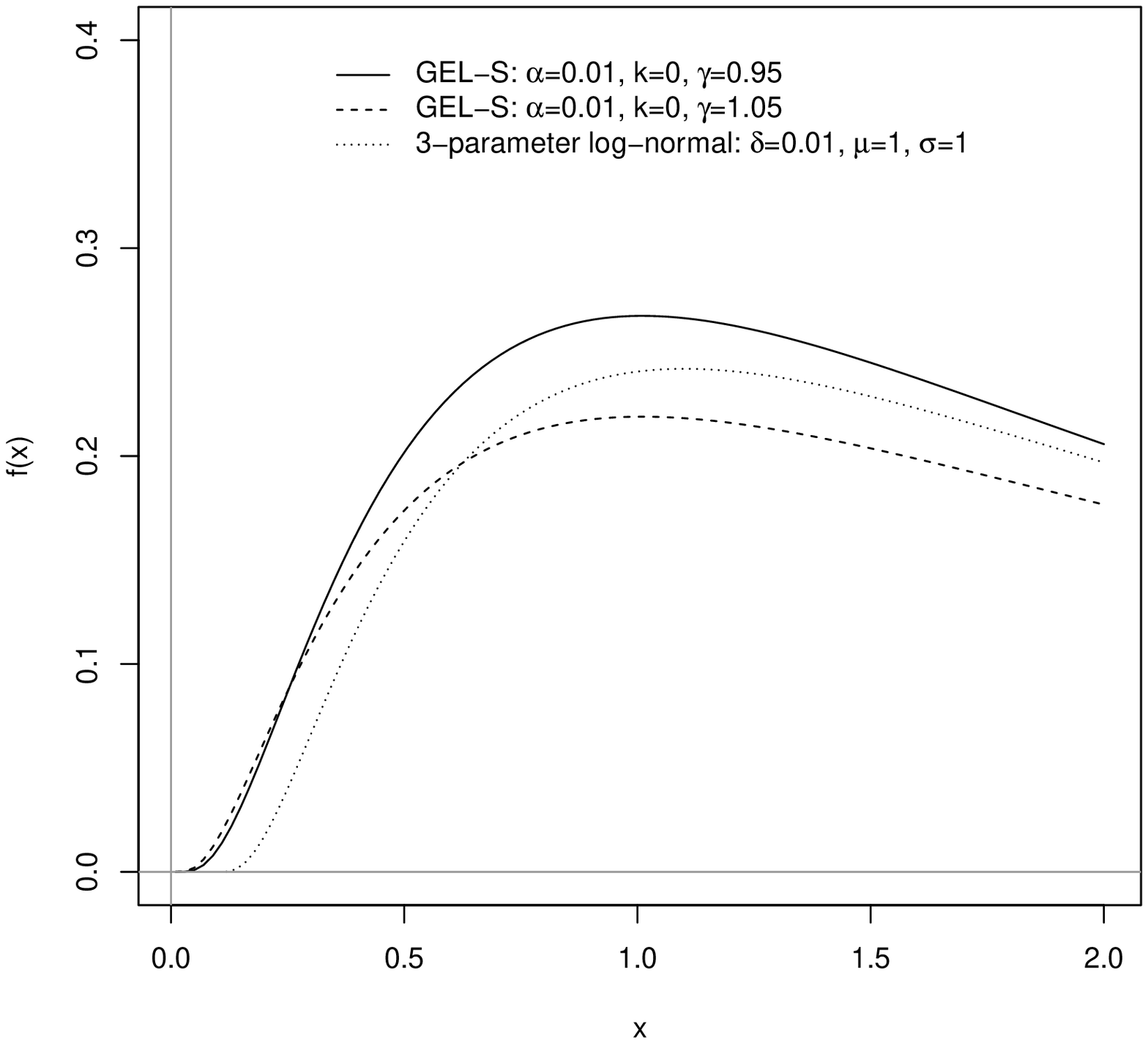} 
\centering \includegraphics[scale=0.30]{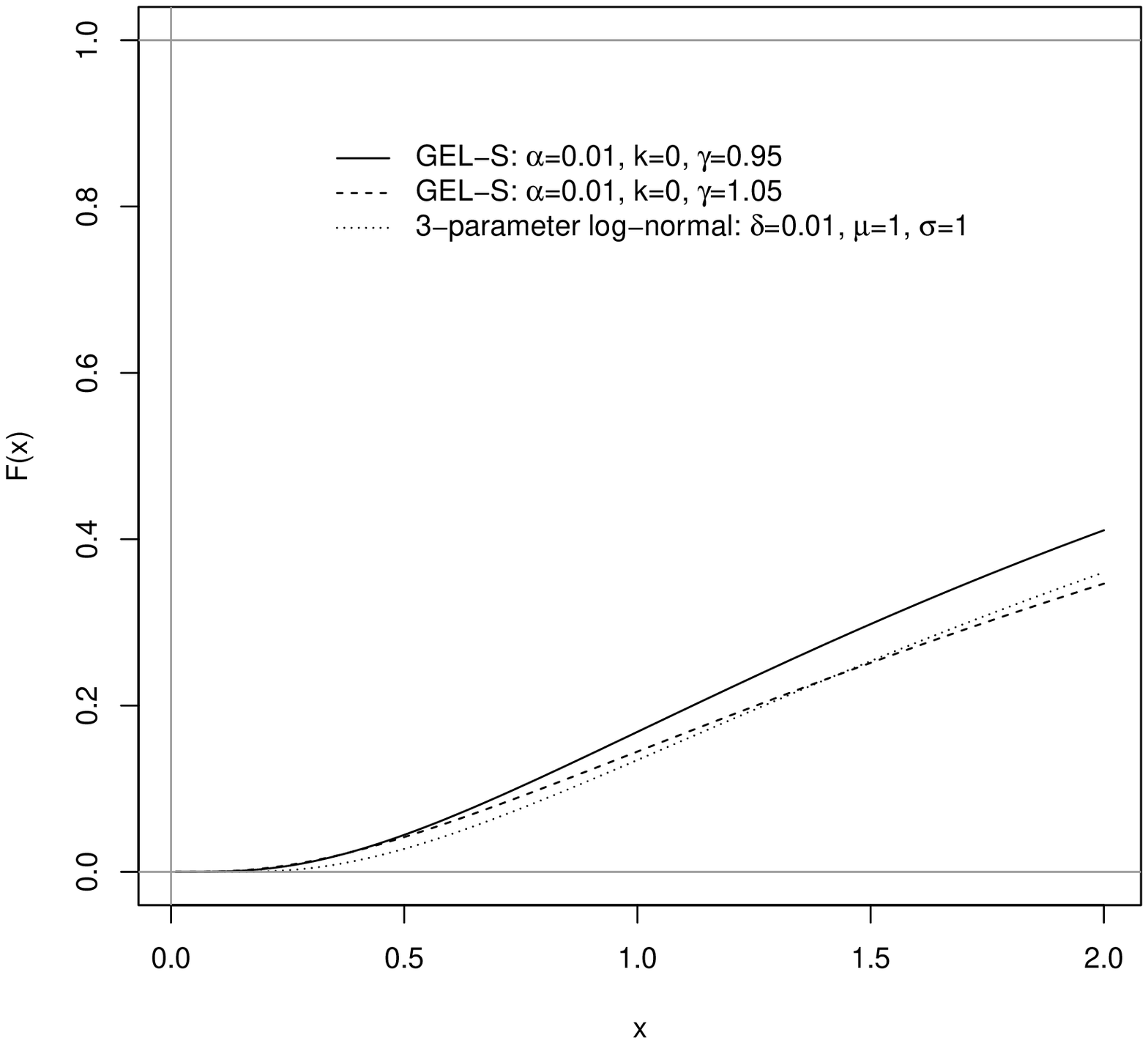} 
\caption{Comparisons of pdfs (left plots) and cdfs (right plots) associated to GEL-S and two-parameter log-normal (top plots) and to GEL-S and three-parameter log-normal (bottom plots) distributions}
\label{fig01} 
\end{figure}

Fig \ref{fig02} presents curves of pdfs and cdfs of GEL-S distributions by varying parameters.
Left plots concern pdfs and right plots their corresponding cdfs.
Each row shows plots where only one parameter varies: $\alpha$ for top plots, $k$ for middle plots, and $\gamma$ for bottom plots.
These plots show that the increase of $\alpha$, $k$, or $\gamma$ always promote the flattening of pdfs.
On the other hand, the increase of $\alpha$ shifts the pdfs and cdfs to the right with slight increases in the heights of the pdfs, whereas
the increase of $\gamma$ increases the right skewness of the pdfs.

\begin{figure}[!ht]
\centering \includegraphics[scale=0.30]{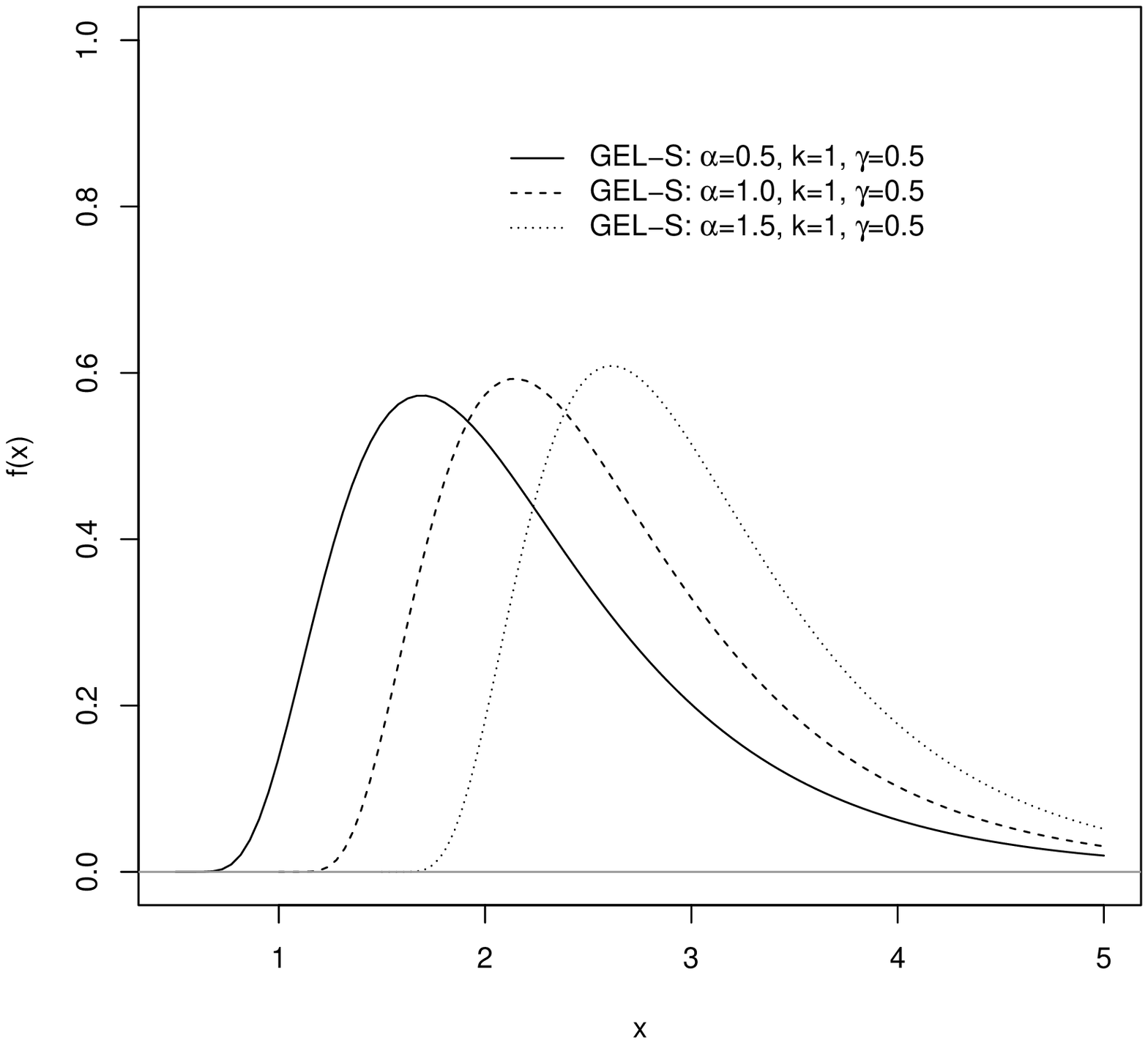} 
\centering \includegraphics[scale=0.30]{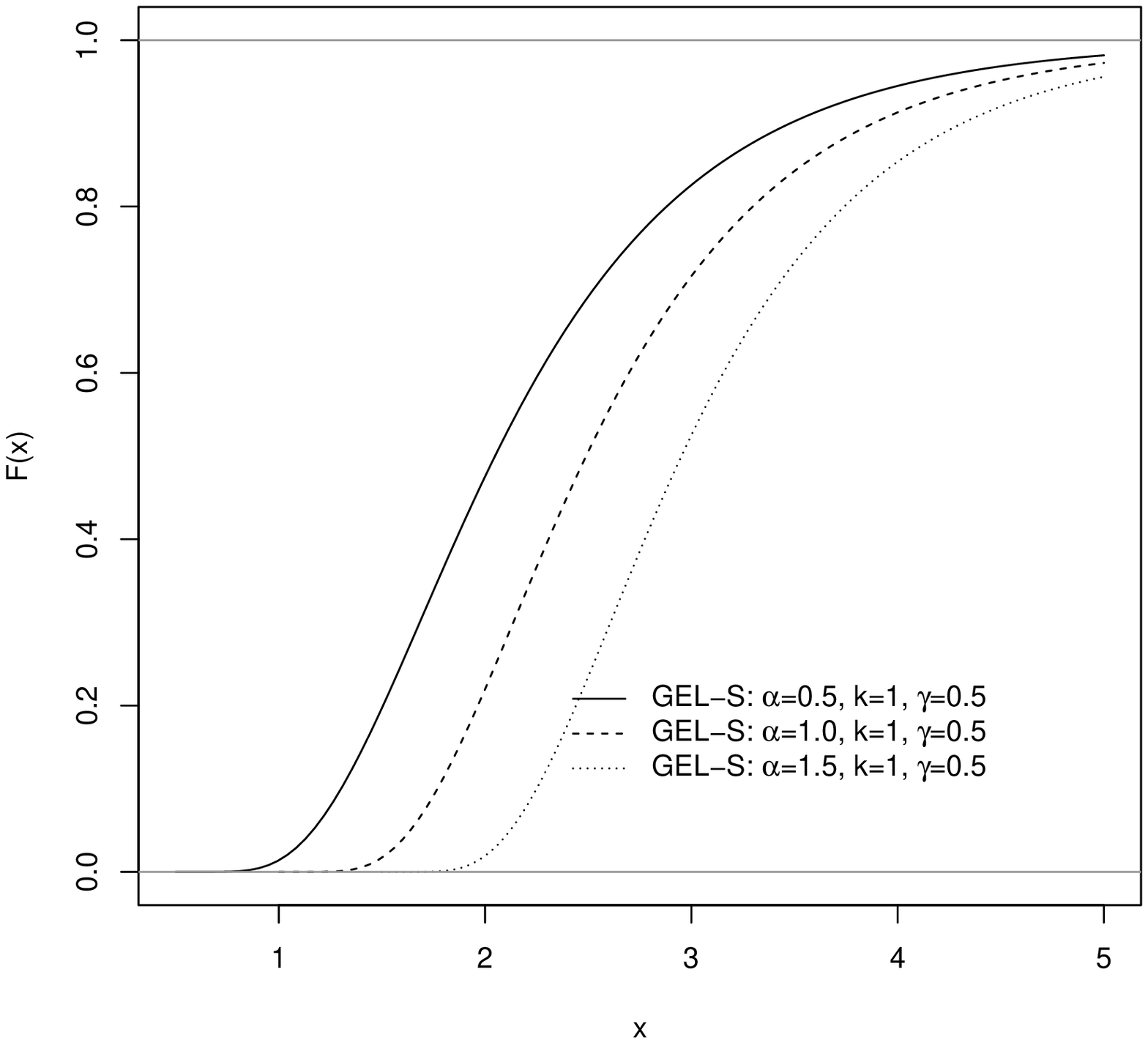} 
\centering \includegraphics[scale=0.30]{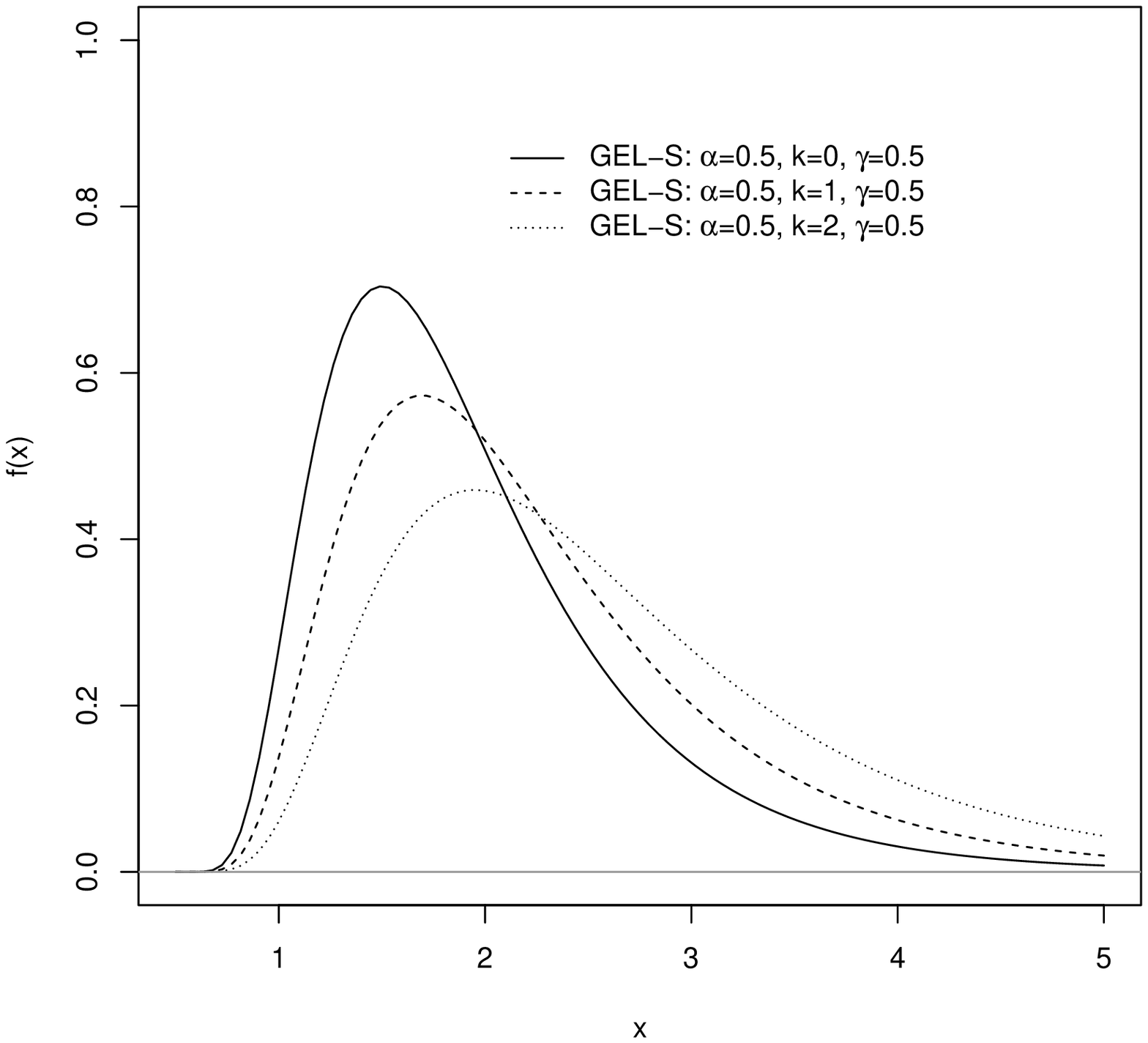} 
\centering \includegraphics[scale=0.30]{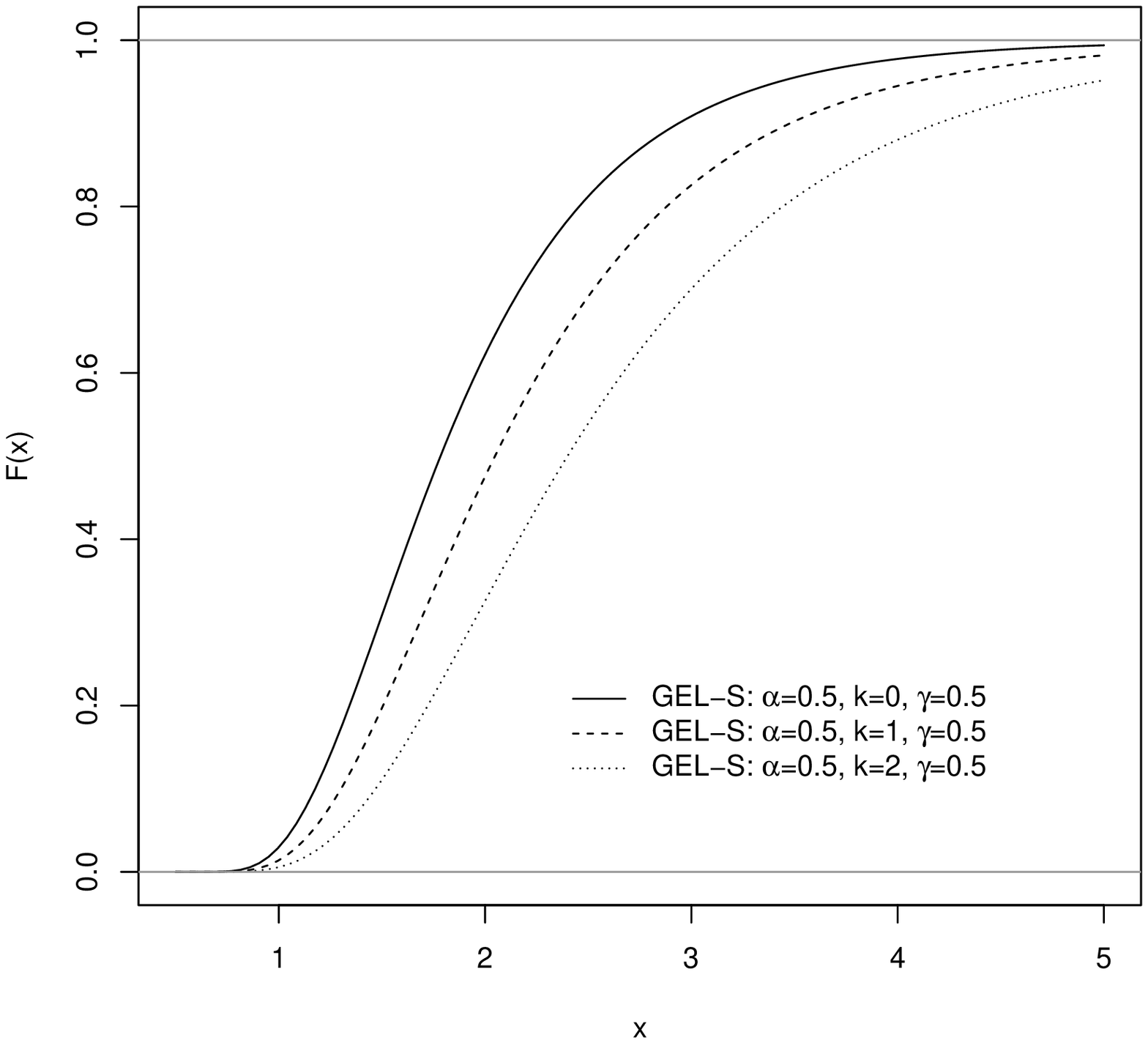} 
\centering \includegraphics[scale=0.30]{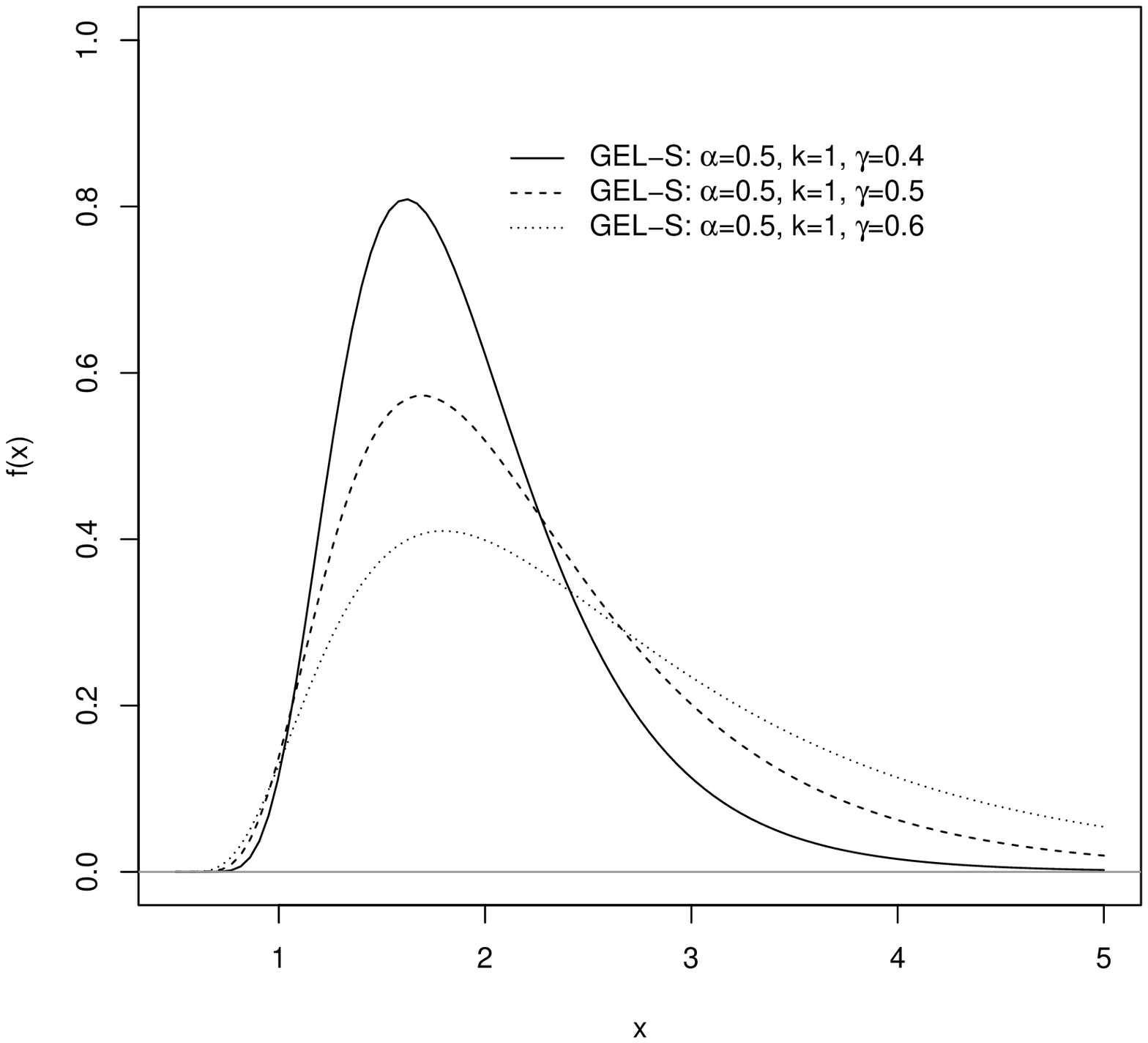} 
\centering \includegraphics[scale=0.30]{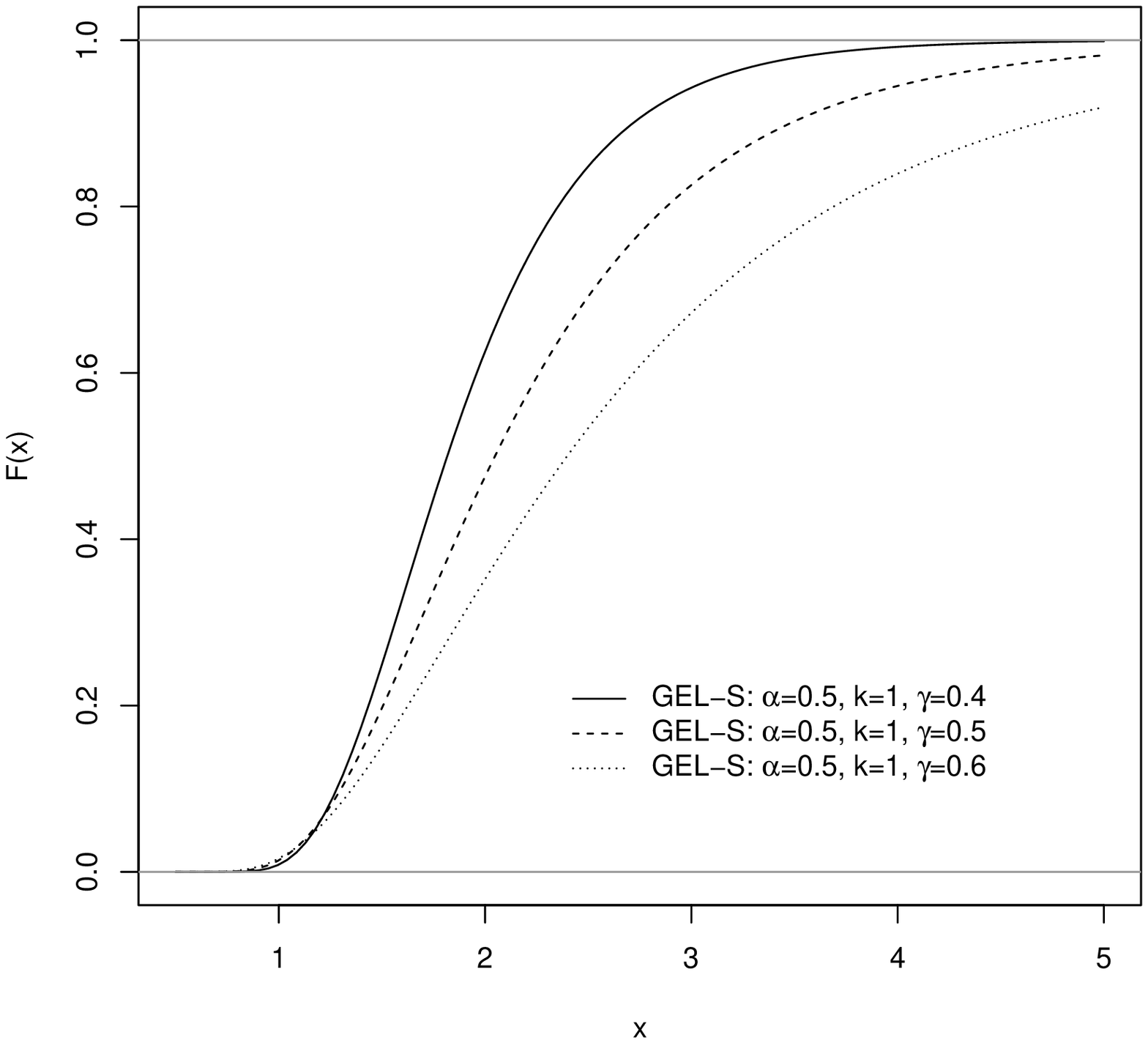} 
\caption{Comparisons of pdfs (left plots) and cdfs (right plots) of GEL-S distributions by varying parameters ($\alpha$ in top plots, $k$ in middle plots, $\gamma$ in bottom plots)}
\label{fig02} 
\end{figure}

\section{Statistical Properties of the GEL-S Distribution}
\label{sec2}

In this section, statistical properties of the GEL-S distribution are studied.
To this aim, hereafter, $X$ denotes a rv following a GEL-S distribution with the parameters $\alpha$, $k$, and $\gamma$, and with the pdf $f$ defined in (\ref{deff}).

\subsection{Mean, Variance, Skewness, Kurtosis, and Moments}

Firstly, the $n$th moment of $X$, $n=0,1,2,\ldots$ is described, computations are presented in Annexe \ref{Proofs}:
\begin{equation}\label{momn}
E\big[X^n\big]:=\int_\alpha^\infty x^n\,f(x)\,dx=C\gamma\sqrt{2\pi}\,\sum_{i=0}^{n+k}{{n+k}\choose{i}}\alpha^{n+k-i}e^{(i+1)^2\gamma^2/2},
\end{equation}
which means that $X$ has all its moments.
From this expression, important statistics of $X$ can be deduced, for instance the mean
$$
\mu_X:=E\big[X\big]=C\gamma\sqrt{2\pi}\,\sum_{i=0}^{1+k}{{1+k}\choose{i}}\alpha^{1+k-i}e^{(i+1)^2\gamma^2/2},
$$
the variance
$$
\sigma^2_X:=E\big[X^2\big]-\big(E\big[X\big]\big)^2
 =C\gamma\sqrt{2\pi}\,\sum_{i=0}^{2+k}{{2+k}\choose{i}}\alpha^{2+k-i}e^{(i+1)^2\gamma^2/2}-\mu_X^2,
$$
the skewness
$$
\textrm{Skew}_X:=E\left[\left(\frac{X-\mu_X}{\sigma_X}\right)^3\right]=\frac{C\gamma\sqrt{2\pi}\,\sum_{i=0}^{3+k}{{3+k}\choose{i}}\alpha^{3+k-i}e^{(i+1)^2\gamma^2/2}-3\mu_X\sigma^2_X-\mu^3_X}{\sigma^3_X}
,
$$
and the kurtosis
$$
\textrm{Kurt}_X:=E\left[\left(\frac{X-\mu_X}{\sigma_X}\right)^4\right]
=\frac{C\gamma\sqrt{2\pi}\,\sum_{i=0}^{4+k}{{4+k}\choose{i}}\alpha^{4+k-i}e^{(i+1)^2\gamma^2/2}-4\mu_X\sigma^3_X\textrm{Skew}_X-6\mu^2_X\sigma^2_X-\mu^4_X}{\sigma^4_X}.
$$
Tab. \ref{Tabxyz} illustrates the previous statistics by considering the GEL-S distributions shown in Fig. \ref{fig02}.
These results show that the increase of the mean, the skewness and the kurtosis are promoted when any of the parameters $\alpha$, $k$ or $\gamma$ increases, but for the variance only the increase of $k$ or $\gamma$ promote its increase.

\begin{table}[!ht]
\centering
\begin{tabular}{ccccr}
\hline
Parameters & $\mu_X$ & $\sigma_X^2$ & $\textrm{Skew}_X$ & \multicolumn{1}{c}{$\textrm{Kurt}_X$} \\
\hline
$\alpha=0.5$, $k=1$, $\gamma=0.5$ & 2.26 & 0.92 & 1.78 & 9.08 \\
$\alpha=1.0$, $k=1$, $\gamma=0.5$ & 2.70 & 0.87 & 1.80 & 9.23 \\
$\alpha=1.5$, $k=1$, $\gamma=0.5$ & 3.16 & 0.84 & 1.81 & 9.33 \\
$\alpha=0.5$, $k=0$, $\gamma=0.5$ & 1.95 & 0.60 & 1.75 & 8.90 \\
$\alpha=0.5$, $k=2$, $\gamma=0.5$ & 2.67 & 1.46 & 1.80 & 9.21 \\
$\alpha=0.5$, $k=1$, $\gamma=0.4$ & 1.93 & 0.37 & 1.34 & 6.33 \\
$\alpha=0.5$, $k=1$, $\gamma=0.6$ & 2.79 & 2.41 & 2.31 & 13.68 \\
\hline
\end{tabular}
\caption{Statistics for the GEL-S distributions shown in Fig. \ref{fig02}}
\label{Tabxyz}
\end{table}

\subsection{Mode}

The explicit expression of $f$ given by (\ref{deff}) allows the analysis of the mode $x_m$ of the GEL-S distribution.
This is given in the following result.

\begin{proposition}\label{PropMode}
The mode of the GEL-S distribution with parameters $\alpha$, $k$ and $\gamma$ exists, is unique and is the solution of the equation
$$
x\,\log(x-\alpha)=k\gamma^2(x-\alpha).
$$
\end{proposition}

The claim on unicity given in the previous proposition shows that the GEL-S distribution is always unimodal.
Furthermore, from the relationship given by this proposition we have that, if $k=0$, $x_m=1+\alpha$, without influence of $\gamma$, whereas if $k>0$, from
$$
x\,(x-\alpha)\,\log(x-\alpha)=k\gamma^2(x-\alpha)^2>0,
$$
$x_m>1+\alpha$ follows.

Illustrations of modes are presented in Tab. \ref{exMode} considering the GEL-S distributions shown in Fig. \ref{fig02}.
Their corresponding means are included.
These results corroborate the relations between the mode and $\alpha$ deduced above.
Also, it is found that the mode is always lower than its corresponding mean, which is in line with the right skewness of the GEL-S distribution.

\begin{table}[!ht]
\centering
\begin{tabular}{ccc}
\hline
Parameters & $\mu_X$ 
& $x_m$ \\
\hline
$\alpha=0.5$, $k=1$, $\gamma=0.5$ & 2.26 
& 1.69 \\
$\alpha=1.0$, $k=1$, $\gamma=0.5$ & 2.70 
& 2.14 \\
$\alpha=1.5$, $k=1$, $\gamma=0.5$ & 3.16 
& 2.61 \\
$\alpha=0.5$, $k=0$, $\gamma=0.5$ & 1.95 
& 1.50 \\
$\alpha=0.5$, $k=2$, $\gamma=0.5$ & 2.67 
& 1.95 \\
$\alpha=0.5$, $k=1$, $\gamma=0.4$ & 1.93 
& 1.62 \\
$\alpha=0.5$, $k=1$, $\gamma=0.6$ & 2.79 
& 1.80 \\
\hline
\end{tabular}
\caption{Means and modes for the GEL-S distributions shown in Fig. \ref{fig02}}
\label{exMode}
\end{table}

\subsection{Quantiles and Random Number Generation}
\label{Quantilesandrandomnumbergeneration}

The quantile function $q(p)$, $0<p<1$, is obtained by solving
$$
F\big(q(p)\big)=p,
$$
therefore, thus, for the GEL-S distribution, this function $q$ corresponds to the solution of the nonlinear equation
\begin{equation}\label{quantile}
\gamma C\sqrt{2\pi}\,\sum_{i=0}^k{{k}\choose{i}}\alpha^{k-i}e^{(i+1)^2\gamma^2/2}\Phi\left(\frac{1}{\gamma}\left(\log(q(p)-\alpha)-(i+1)\gamma^2\right)\right)=p.
\end{equation}
Since
$$
F'(x)=C\,\frac{1}{x-\alpha} \sum_{i=0}^k{{k}\choose{i}}\alpha^{k-i}e^{(i+1)^2\gamma^2/2}e^{-\frac{1}{2}\left(\frac{1}{\gamma}\left(\log(x-\alpha)-(i+1)\gamma^2\right)\right)^2}>0,\quad x>\alpha,
$$
we have that the solution of (\ref{quantile}) is unique.

Illustrations of quantiles are presented in Tab. \ref{exQuantile}.
To compute quantiles, i.e. to solve (\ref{quantile}), the function \texttt{uniroot} in the R software package was used.
This table shows the quantile when $p=0.5$, i.e. the median of $X$, $x_M$, for the distributions presented in Fig. \ref{fig02}.
Means taken from Tab. \ref{Tabxyz} are included in Tab. \ref{exQuantile} in order to compare all these statistics.
The quantiles $q(0.01)$, $q(0.05)$, $q(0.95)$ and $q(0.99)$ are also incorporated to Tab. \ref{exQuantile}, which may be used as risk measures in contexts like insurance or finance 
\cite{AlexanderSarabia2012,BellesSamperaGuillenSantolino2016}.
These results show that in all cases the medians are lower than the means, this means that the bulk of data is concentrated to the left of the mean which is in line with the right skewness of this type of distributions.
Also, as expected, $q(p)$ is increasing in $p$ and $q(0.01)$ is near to $\alpha$, whereas due to the right skewness of the GEL-S distribution, the differences between $q(0.05)$ and $q(0.01)$ are lower than the ones between $q(0.99)$ and $q(0.95)$.

\begin{table}[!ht]
\centering
\begin{tabular}{cccccccc}
\hline
Parameters & $\mu_X$ & $q(0.5)$ or $x_M$ & $q(0.01)$ & $q(0.05)$ & $q(0.95)$ & $q(0.99)$ \\
\hline
$\alpha=0.5$, $k=1$, $\gamma=0.5$ & 2.26 & 2.05 & 0.97 & 1.17 & 4.08 & 5.56 \\
$\alpha=1.0$, $k=1$, $\gamma=0.5$ & 2.70 & 2.49 & 1.45 & 1.64 & 4.47 & 5.92 \\
$\alpha=1.5$, $k=1$, $\gamma=0.5$ & 3.16 & 2.95 & 1.94 & 2.12 & 4.89 & 6.31 \\
$\alpha=0.5$, $k=0$, $\gamma=0.5$ & 1.95 & 1.78 & 0.90 & 1.06 & 3.42 & 4.61 \\
$\alpha=0.5$, $k=2$, $\gamma=0.5$ & 2.67 & 2.40 & 1.06 & 1.30 & 4.96 & 6.83 \\
$\alpha=0.5$, $k=1$, $\gamma=0.4$ & 1.93 & 1.87 & 1.01 & 1.17 & 3.07 & 3.88 \\
$\alpha=0.5$, $k=1$, $\gamma=0.6$ & 2.79 & 2.40 & 0.95 & 1.18 & 5.72 & 8.42 \\
\hline
\end{tabular}
\caption{Means and quantiles for the GEL-S distributions shown in Fig. \ref{fig02}}
\label{exQuantile}
\end{table}

The solution $q$ of (\ref{quantile}) given $p$, $0<p<1$, could be used to generate random numbers of a rv that follows a GEL-S distribution.
Indeed, since $F'>0$, the (non-explicit) function $F^{-1}(p)$ is strictly increasing and then, the inverse transform sampling method can be applied to draw random samples.
This method consists in \cite{Devroye}
\begin{enumerate}
\item
Generate a random number $p$ from the standard uniform distribution in the interval $[0,1]$; and,

\item
Compute $q$ such that $F(q)=p$, i.e. (\ref{quantile}).

\end{enumerate}

The implementation of the previous method may be done by generating random numbers following a uniform distribution that may be performed using the function \texttt{runif} in the R software package, and thereafter, by computing quantiles that may be performed using the function \texttt{uniroot} mentioned above.

This random number generation procedure will be used later on in order to simulate random numbers following a GEL-S distribution.
These numbers will be used to study the performance of the new distribution.

\section{Maximum Likelihood Estimation}
\label{sec3}

In this section, 
the method of maximum likelihood for estimating $\alpha$, $k$ and $\gamma$ of the GEL-S distribution is proposed.

Let $X$ be a rv following a GEL-S distribution with parameters $\alpha$, $k$ and $\gamma$, and let $x_1$, \ldots, $x_n$ be a sample of $X$ obtained independently.
Let $\theta=(\alpha,k,\gamma)$.

Following the method of maximum likelihood, the likelihood function of this random sample is then given by
$$
L(\theta|x_1,\ldots,x_n)=\prod_{i=1}^n
C\,x_i^k e^{-(2\gamma^2)^{-1}\left(\log(x_i-\alpha)\right)^2},
$$
and then its log-likelihood function is
$$
l(\theta|x_1,\ldots,x_n)=n\,\log C+k\sum_{i=1}^n\log x_i-\frac{1}{2\gamma^2}\sum_{i=1}^n\left(\log(x_i-\alpha)\right)^2.
$$
Maximum likelihood estimates (MLEs) of $\alpha$, $k$ and $\gamma$ might be reached by solving the non-linear system obtained by equaling to 0 the derivatives of $l$ with respect to $\theta$.
Unfortunately, the parameter $k$ is not continuous and thus such procedure cannot be applied.

In order to overcome this issue, the following \emph{ad-hoc} alternative to reach the maximum of $l$ is proposed.
Fixing $k=0,1,2,\ldots$, $l$ is maximized by searching optimal estimates $\hat{\alpha}(k)$ and $\hat{\gamma}(k)$.
Then, $k$, $\hat{\alpha}(k)$ and $\hat{\gamma}(k)$ are selected as the ones that maximize $l$ through the range of values $k$ taken into account.
This procedure is equivalent to maximize $l$ throughout the three parameters, but only considering a few values for $k$.
Then, following this proposed procedure we need to solve the non-linear system, fixed $k$,
\begin{eqnarray*}
\frac{\partial l}{\partial\alpha} & = & n\frac{1}{C}\frac{\partial C}{\partial\alpha}+\frac{1}{\gamma^2}\sum_{i=1}^n\frac{\log(x_i-\alpha)}{x_i-\alpha}\quad=\quad0 \\
\frac{\partial l}{\partial\gamma} & = & n\frac{1}{C}\frac{\partial C}{\partial\gamma}+\frac{1}{\gamma^3}\sum_{i=1}^n\left(\log(x_i-\alpha)\right)^2\quad=\quad0.
\end{eqnarray*}
There are no explicit solutions for this system.
A method to numerically solve such a system is the Newton-Raphson (NR) algorithm.
This is a well-known and useful technique for finding roots of systems of non-linear equations in several variables. 
The function \texttt{nlm} in the R software package, that carries out a minimization of an objective function using a NR-type algorithm, is used to solve the system described above.
To this aim, this function is applied to the objective function $-l(\theta|x_1,\ldots,x_n)$ given $k$, which provides MLEs $\hat{\theta}$ of $\theta$.

A limitation of the function \texttt{nlm} is that constraints are not allowed.
This is an issue for estimating both parameters $\alpha$ and $\gamma$ of a GEL-S distribution since $\alpha$ needs to be non-negative and $\gamma$ positive, i.e. negative values as estimates for $\alpha$ and $\gamma$ are not allowed.
In practice, applications of \texttt{nlm} to get estimates for $\alpha$ and $\gamma$ showed that only the estimates of $\alpha$ could eventually be negative.
In order to circumvent this limitation, the following simple modification of $\alpha$ in the objective function to be minimized could be used: consider $\alpha^2$ instead of $\alpha$.
This means that $\alpha$ could be estimated by negative values, but then the true value for $\alpha$ is positive since it is equal to $\alpha^2$.

For interval estimation of $(\alpha,\gamma)$ and hypothesis tests on these parameters, we use the $2\times2$ observed information matrix given by, fixed $k$,
$$
I(\theta)
=-E\left[
\begin{array}{cc}
\displaystyle \frac{\partial^2 l}{\partial\alpha^2} & \displaystyle \frac{\partial^2 l}{\partial\alpha\partial\gamma} \\
 \vspace{-2mm} & \\
\displaystyle \frac{\partial^2 l}{\partial\alpha\partial\gamma} & \displaystyle \frac{\partial^2 l}{\partial\gamma^2}
\end{array}
\right]
$$
where
\begin{eqnarray*}
\frac{\partial^2 l}{\partial\alpha^2} & = & n\frac{1}{C}\frac{\partial^2 C}{\partial\alpha^2}-n\frac{1}{C^2}\left(\frac{\partial C}{\partial\alpha}\right)^2+\frac{1}{\gamma^2}\sum_{i=1}^n\left(\frac{\log(x_i-\alpha)}{(x_i-\alpha)^2}-\frac{1}{(x_i-\alpha)^2}\right) \\
\frac{\partial^2 l}{\partial\alpha\partial\gamma} & = & n\frac{1}{C}\frac{\partial^2 C}{\partial\alpha\partial\gamma}-n\frac{1}{C^2}\frac{\partial C}{\partial\alpha}\frac{\partial C}{\partial\gamma}-\frac{2}{\gamma^3}\sum_{i=1}^n\frac{\log(x_i-\alpha)}{x_i-\alpha} \\
\frac{\partial^2 l}{\partial\gamma^2} & = & n\frac{1}{C}\frac{\partial^2 C}{\partial\gamma^2}-n\frac{1}{C^2}\left(\frac{\partial C}{\partial\gamma}\right)^2-\frac{3}{\gamma^4}\sum_{i=1}^n\left(\log(x_i-\alpha)\right)^2.
\end{eqnarray*}
Under certain regularity conditions \cite{Cramer}, the MLE $\hat{\theta}$ given $k$ approximates as $n$ increases a multivariate normal distribution with
mean equal to the true parameter value $\theta$ and variance-covariance matrix given
by the inverse of the observed information matrix, i.e. $\Sigma=\big[\sigma_{ij}\big]=I^{-1}(\theta)$.
Hence, the asymptotic behavior of two-sided $(1-\epsilon)100~\%$ confidence intervals (CIs) for the
parameters $\alpha$ and $\gamma$ are approximately
$$
\hat{\alpha}\pm z_{\epsilon/2}\sqrt{\hat{\sigma}_{11}},\quad
\hat{\gamma}\pm z_{\epsilon/2}\sqrt{\hat{\sigma}_{22}}
$$
where $z_\delta$ represents the $\delta\,100~\%$ percentile of the standard normal distribution.

\section{Simulation Studies}
\label{sec4}

In this section, Monte Carlo simulation studies are carried out to assess the performance of the MLEs of $\alpha$ and $\gamma$ described in the previous section.
Two sets of parameters are considered, each one corresponding to one study.
The true parameters for these studies are presented in Tab. \ref{trueValues}.

\begin{table}[!ht]
\centering
\begin{tabular}{cccc}
\hline
Study & $\alpha$ & $k$ & $\gamma$ \\
\hline
I & 1.0 & 2 & 1.0 \\
II & 2.0 & 4 & 0.5 \\
\hline
\end{tabular}
\caption{Parameters for simulation studies}
\label{trueValues}
\end{table}

Each study takes into account the following scenarios by varying the sample size $n$: 1\,000 and 10\,000.
Then, following the procedure to generate random numbers indicated in Subection \ref{Quantilesandrandomnumbergeneration}, random numbers are simulated from a GEL-S distribution with given parameters $\alpha$, $k$ and $\gamma$.
A fixed seed is used to generate such random numbers, impliying that all results of these studies can always be exactly replicated.

Fig. \ref{figHist} exhibits histograms of the empirical pdfs of the samples analyzed.
These plots are built using 100 bins in order to have enough detail on the shape of these empirical curves.
The plots on top correspond to the study I and the ones on bottom to the study II.
From these plots, a greater right skewness for data of the study I than the one for data of the study II is observed, independently of variations of $n$.

\begin{figure}[!ht]
\centering \includegraphics[scale=0.30]{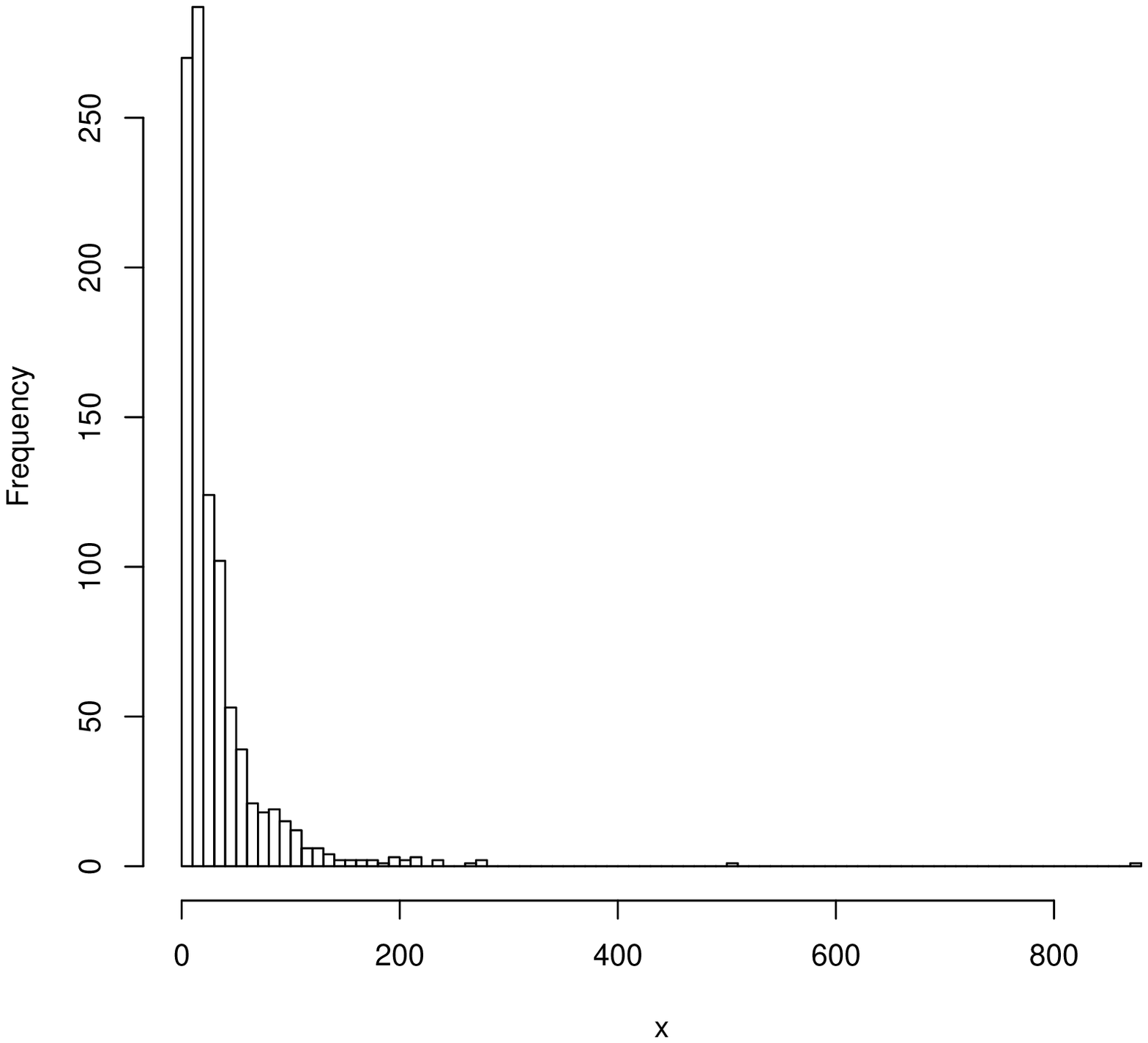} 
\centering \includegraphics[scale=0.30]{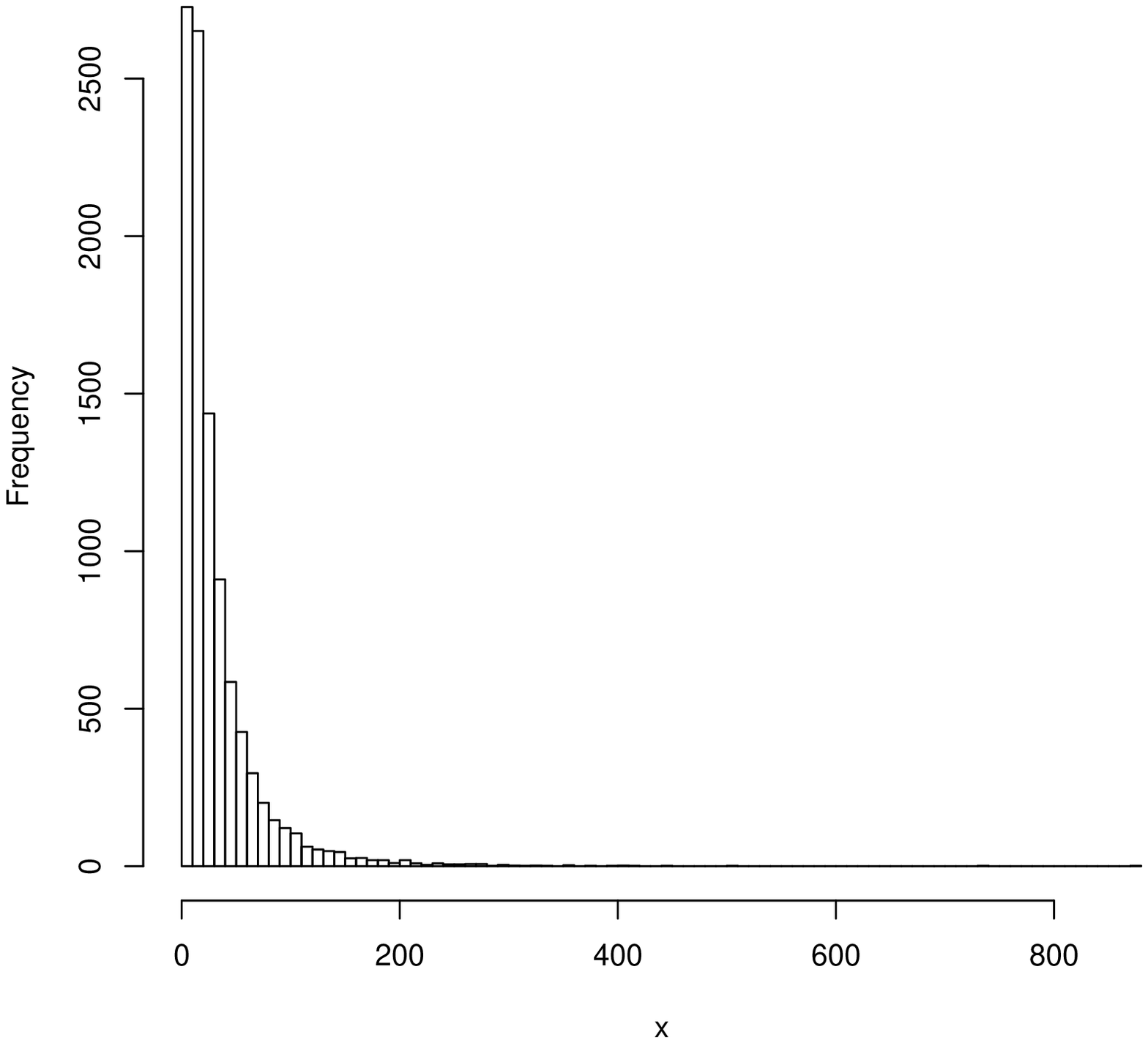} 
\centering \includegraphics[scale=0.30]{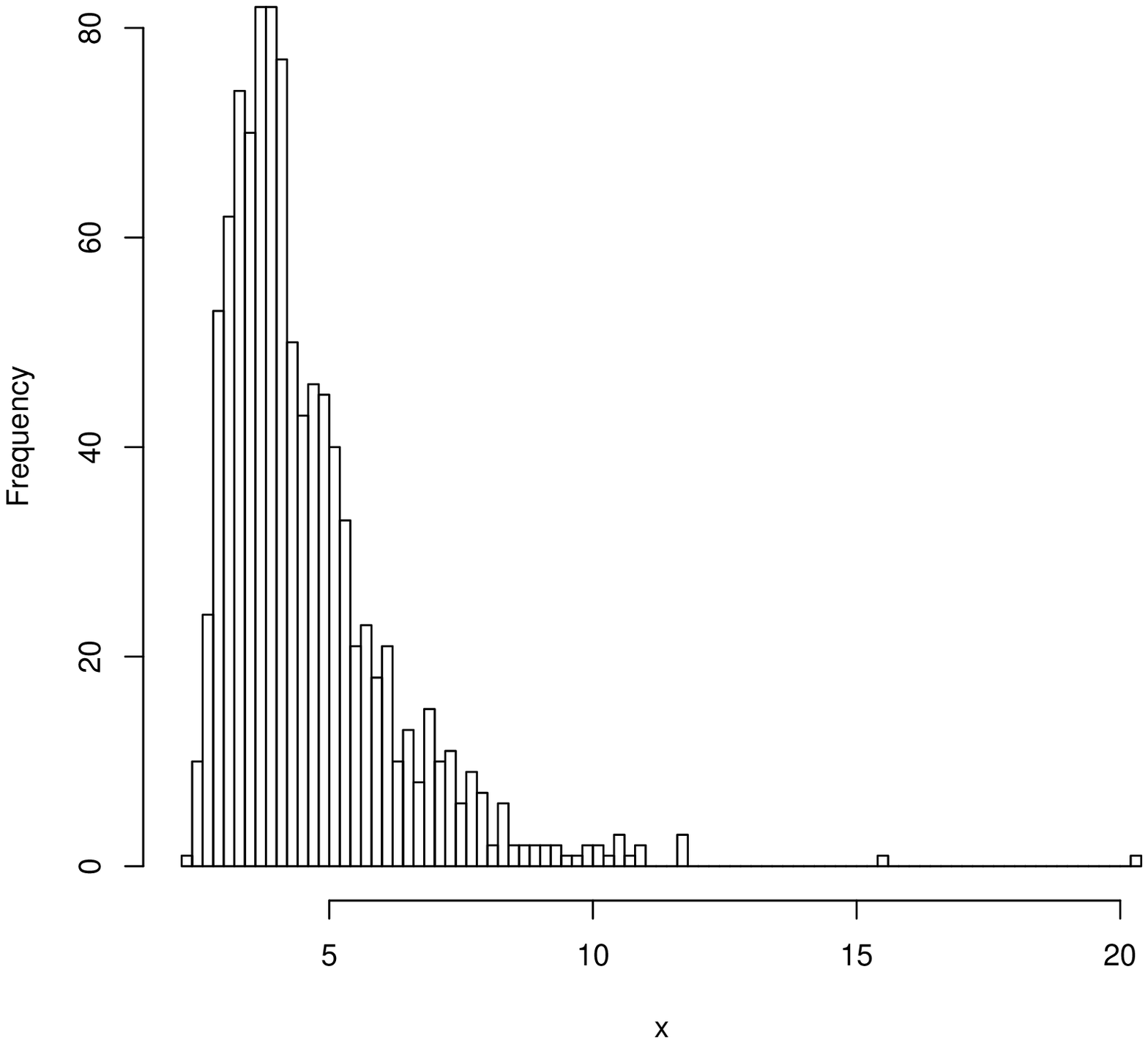} 
\centering \includegraphics[scale=0.30]{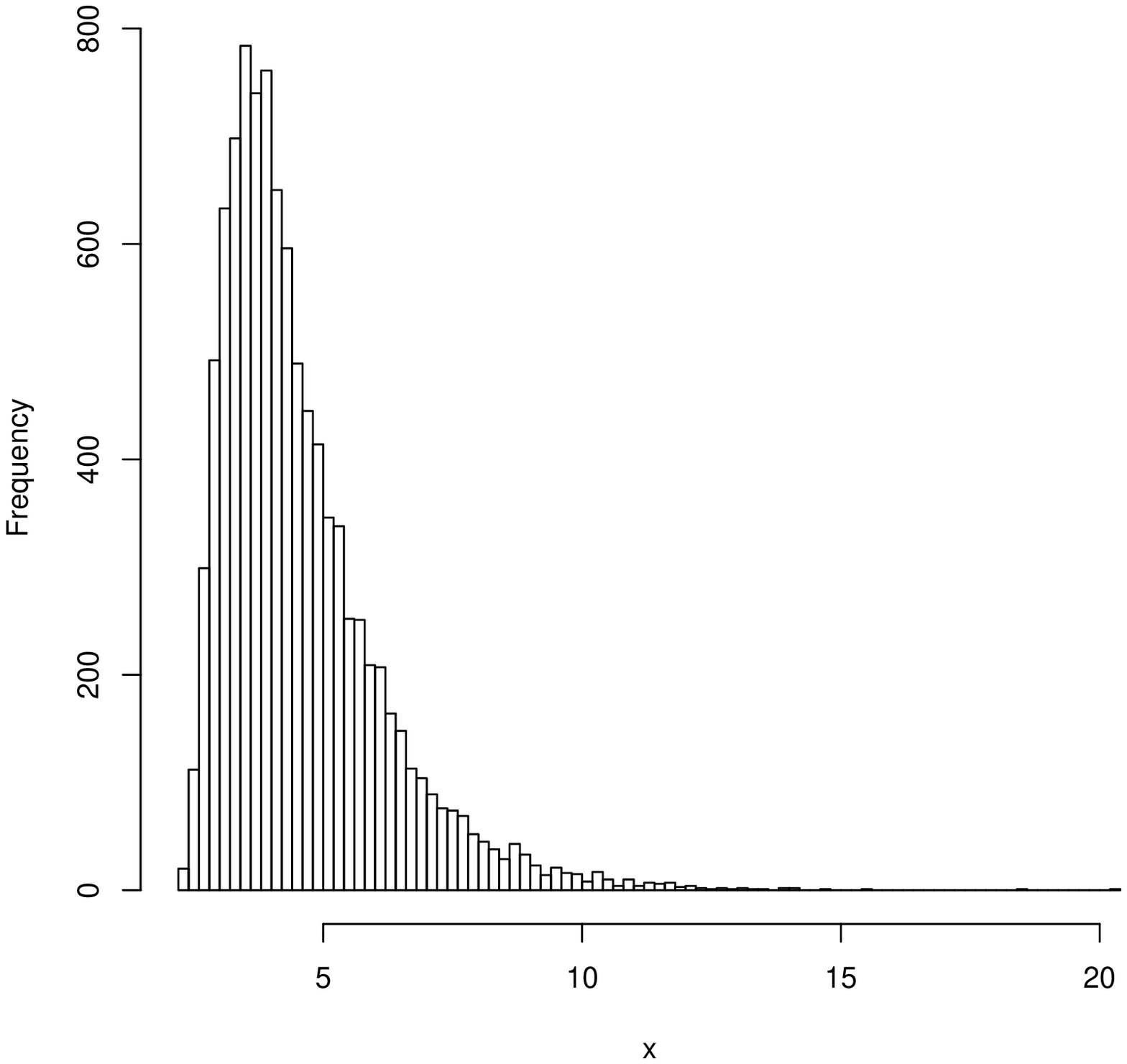} 
\caption{Histograms by varying the parameters of the GEL-S distribution (study I on top and study II on bottom) and by varying $n$ ($n=1\,000$ to the left and $n=10\,000$ to the right)}
\label{figHist} 
\end{figure}

Next, estimates of $\alpha$ and $\gamma$ are computed given $k$, using the procedure proposed in Section \ref{sec3} for estimating $\alpha$ and $\gamma$ given $k$.
Considering always ranges of $k$ from 0 to 6, 
Tab. \ref{ParameterestimatesSimulation} shows these results by varying the true parameters and $n$.
For each $k$, the observed maximum likelihood is included. 
Then, by study and $n$, the models with the highest likelihood over the studied range of $k$ are selected.
These selected models are highlighted.
It is found that the values $k$ of the selected models correspond to the true values $k$, except when $n=1\,000$ in study II.
Hence, 
it seems that, under the estimate method proposed, 
for small samples with not so high skewness, other than the true parameter $k$ could be possible.
On the estimates of $\gamma$ given by the selected models, they are the nearest to the true parameters, except when $n=1\,000$ in study II.
Considering $\hat{\alpha}$ of the selected models, they are not always the nearest to the true parameters.

\begin{table}[!ht]
\centering
\begin{tabular}{cccc}
\multicolumn{4}{c}{$n=1\,000$} \\
\hline
Given & \multicolumn{2}{c}{Estimates} & Maximum \\
\cline{2-3}
$k$ & $\hat{\alpha}$ & $\hat{\gamma}$ & likelihood \\
\hline
0 & 1.477 & 1.599 & $-3524$ \\
1 & 1.381 & 1.601 & $-3438$ \\
\textbf{2} & \textbf{1.190} & \textbf{1.004} & $\mathbf{-3416}$ \\
3 & 1.121 & 0.876 & $-3435$ \\
4 & 1.156 & 0.787 & $-3480$ \\
5 & 1.195 & 0.721 & $-3543$ \\
6 & 1.225 & 0.669 & $-3619$ \\
\hline
\end{tabular}
\hspace{5mm}
\begin{tabular}{cccc}
\multicolumn{4}{c}{$n=10\,000$} \\
\hline
Given & \multicolumn{2}{c}{Estimates} & Maximum \\
\cline{2-3}
$k$ & $\hat{\alpha}$ & $\hat{\gamma}$ & likelihood \\
\hline
0 & 1.204 & 1.601 & $-35355$ \\
1 & 1.171 & 1.205 & $-34397$ \\
\textbf{2} & \textbf{1.003} & \textbf{1.001} & $\mathbf{-34182}$ \\
3 & 0.955 & 0.873 & $-34394$ \\
4 & 1.002 & 0.785 & $-34881$ \\
5 & 1.043 & 0.719 & $-35549$ \\
6 & 1.072 & 0.667 & $-36346$ \\
\hline
\end{tabular}
\begin{tabular}{cccc}
\multicolumn{4}{c}{} \\
\multicolumn{4}{c}{$n=1\,000$} \\
\hline
Given & \multicolumn{2}{c}{Estimates} & Maximum \\
\cline{2-3}
$k$ & $\hat{\alpha}$ & $\hat{\gamma}$ & likelihood \\
\hline
0 & 2.309 & 0.732 & $-723$ \\
1 & 2.244 & 0.646 & $-714$ \\
2 & 2.171 & 0.585 & $-710$ \\
\textbf{3} & \textbf{2.097} & \textbf{0.538} & $\mathbf{-709}$ \\
4 & 2.021 & 0.500 & $-710$ \\
5 & 1.945 & 0.469 & $-712$ \\
6 & 1.868 & 0.444 & $-715$ \\
\hline
\end{tabular}
\hspace{5mm}
\begin{tabular}{cccc}
\multicolumn{4}{c}{} \\
\multicolumn{4}{c}{$n=10\,000$} \\
\hline
Given & \multicolumn{2}{c}{Estimates} & Maximum \\
\cline{2-3}
$k$ & $\hat{\alpha}$ & $\hat{\gamma}$ & likelihood \\
\hline
0 & 2.253 & 0.737 & $-7391$ \\
1 & 2.206 & 0.649 & $-7256$ \\
2 & 2.138 & 0.587 & $-7197$ \\
3 & 2.064 & 0.539 & $-7171$ \\
\textbf{4} & \textbf{1.988} & \textbf{0.501} & $\mathbf{-7164}$ \\
5 & 1.912 & 0.470 & $-7171$ \\
6 & 1.834 & 0.444 & $-7186$ \\
\hline
\end{tabular}
\caption{Parameter estimates in studies I (top) and II (bottom) given $k$ (the selected models are highlighted)}
\label{ParameterestimatesSimulation}
\end{table}

For the estimates of ${\alpha}$ and ${\gamma}$ indicated in the selected models in Tab. \ref{ParameterestimatesSimulation},
Tab. \ref{CISimulations} reports their 95~\% CIs computed using standard errors of the MLEs of ${\alpha}$ and ${\gamma}$ computed from the observed Hessian matrix provided by the function \texttt{nlm}.
These results show that the errors of these estimates, as expected, decrease when $n$ increases, and it seems that the errors of $\hat{\gamma}$ are systematically lower than the ones of $\hat{\alpha}$.

\begin{table}[!ht]
\centering
\begin{tabular}{cccc}
\hline
Study & $n$ & ${\alpha}$ & ${\gamma}$ \\
\hline
I & 1\,000 & $1.190\pm0.279$ & $1.004\pm0.011$ \\
 & 10\,000 & $1.003\pm0.099$ & $1.001\pm0.003$ \\
\cline{1-4}
II & 1\,000 & $2.097\pm0.051$ & $0.538\pm0.010$ \\
 & 10\,000 & $1.988\pm0.018$ & $0.501\pm0.003$ \\
\hline
\end{tabular}
\caption{95~\% CIs for $\alpha$ and $\gamma$ in the simulation studies}
\label{CISimulations}
\end{table}

\section{Applications}
\label{sec5}

In this section, we present applications in order to illustrate the
performance and usefulness of the proposed distribution when compared to natural competitors.

Popular right-skewed real data from several domains are used.
In all cases, these data have been analyzed in other researches and, in this paper, are fitted using the GEL-S distribution.
This allows the immediate comparison of our results with respect to the ones of competitors.

GEL-S parameters are always estimated using the procedure of maximum likelihood described in Section \ref{sec3}.

Through all these applications, models are compared using the Akaike information criterion (AIC), this criterion being defined by $-2\,n_p-2\,l$ where $n_p$ is the number of parameters of the model, and the Schwarz information criterion (SIC) also called the Bayesian information criterion, this criterion being defined by $n_p\, \log n -2l$ with $n$ the sample size.
For both criteria, the lower the better.

\subsection{Leukaemia Data}

In the first application, uncensored survival times in weeks of 33 patients who died from acute myelogenous leukaemia reported and analyzed by Feigl and Zelen \cite{FeiglZelen1965} are considered.
Other characteristics associated to that illness of those individuals are also available.
One of these characteristics is the presence or absence of a morphologic characteristic of white blood cells, which was used by Feigl and Zelen to factor individuals in order to develop models for each one of those two groups.
The data are available in e.g. the package MASS in software R.

The MLEs for the parameters of the GEL-S distribution for the studied leukaemia data and the corresponding reached log-likelihood are presented in Tab. \ref{leukaemiaGELS}.
Standard errors in brackets.

\begin{table}[!ht]
\centering
\begin{tabular}{cccc}
\hline
$k$ & $\hat{\alpha}^2$ & $\hat{\gamma}$ & $-l$ \\
\hline
0 (fixed) & $0.7958_{(0.1889)}$ & $1.6504_{(0.0800)}$ & 153.24 \\
\hline
\end{tabular}
\caption{Fit of leukaemia data using the GEL-S distribution}
\label{leukaemiaGELS}
\end{table}

The studied leukaemia data have been analyzed by others authors, namely to build regressions \cite{VenablesRipley}.
Mead et al. \cite{MeadAfifyHamedaniGhosh2017} fitted the beta exponential Fr\'{e}chet (BExFr) distribution to these data.
This distribution based on the Fr\'{e}chet (F) distribution incorporates three extra shape parameters.
These authors also presented fits to those data using the F distribution and two other extensions of this distribution, the beta Fr\'{e}chet (BFr) \cite{NadarajahGupta2004} and exponentiated Fr\'{e}chet (EFr) \cite{NadarajahKotz2003} distributions, and also fitted other models: the generalized
inverse gamma (GIG) \cite{Mead2015}, McDonald Lomax (McL) \cite{LemonteCordeiro2013}, Zografos-Balakrishnan log-logistic (ZBLL) \cite{ZografosBalakrishnan2009} and gamma
Lomax (GL) \cite{CordeiroOrtegaPopovic2015}
distributions.
These authors presented the log-likelihoods associated to all those models, which we used it to compute their corresponding AIC and SIC values that are presented in Tab. \ref{leukaemiaTab}.

Tab. \ref{leukaemiaTab} also presents both the AIC and SIC values associated to the GEL-S model when the parameters indicated in Tab. \ref{leukaemiaGELS} are considered.
Since models with lower AIC and SIC values are preferable (the highlighted ones), both the AIC and the SIC favor the GEL-S model overall.

\begin{table}[!ht]
\centering
\begin{tabular}{lccc}
\hline
\multicolumn{1}{c}{Model} & $n_p$ $^{(*)}$ & AIC & SIC \\
\hline
BExFr & 5 & 318.11 & 325.59 \\
BFr & 4 & 316.21 & 321.81 \\
EFr & 3 & 316.89 & 321.09 \\
Fr & 2 & 315.99 & 318.79 \\
GEL-S & 2 & $\mathbf{310.49}$ & $\mathbf{313.48}$ \\
GIG & 5 & 318.23 & 325.23 \\
GL & 3 & 317.03 & 321.24 \\
McL & 5 & 320.95 & 327.96 \\
ZBLL & 3 & 324.82 & 329.02 \\
\hline
\multicolumn{4}{l}{$^{(*)}$ $n_p$: number of parameters}
\end{tabular}
\caption{Leukaemia data: AIC and SIC values}
\label{leukaemiaTab}
\end{table}

Fig. \ref{leukaemiaFood} presents a histogram with 25 bins of the leukaemia data used and the overlay of the fitted GEL-S model reached.

\begin{figure}[!ht]
\centering \includegraphics[scale=0.40]{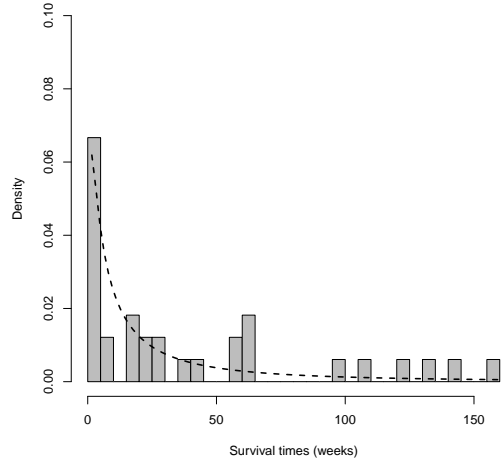} 
\caption{Leukaemia data: Histogram and fitted GEL-S model (dashed line)}
\label{leukaemiaFood} 
\end{figure}

\subsection{Strength Data}

In the second application, strength data reported by Badar and Priest \cite{BaderPriest1982} are considered.
These data consist in strengths measured in giga-Pascals (GPa), for single carbon fibers and for impregnated
1000-carbon fiber tows. These materials were tested under tension at different gauge lengths: 1, 10, 20,
and 50 mm for single fibers and 20, 50, 150 and
300 mm for fiber tows.
Details on the experiment are given in \cite{DurhamPadgett1997}.
These data have been used to mainly perform inference for $P(X<Y)$ where $X$ and $Y$ represent strengths for single fibers at two different gauge lengths \cite{BaderPriest1982,SurlesPadgett1998,SurlesPadgett2001,RaqabKundu2005,KunduGupta2006,KunduRaqab2009}.
Other authors have fitted known distributions to some of these data sets. These models are mentioned below.

In this subsection the data set corresponding to strengths on single fibers tested under tension at the gauge length of 10 mm is examined.
The MLEs for the parameters of the GEL-S distribution for these data and the corresponding reached log-likelihood are presented in Tab. \ref{strengthGELS}.
Standard errors in brackets.

\begin{table}[!ht]
\centering
\begin{tabular}{cccc}
\hline
$k$ & $\hat{\alpha}^2$ & $\hat{\gamma}$ & $-l$ \\
\hline
17 (fixed) & $0.8868_{(0.1034)}$ & $0.2415_{(0.0050)}$ & 56.20 \\
\hline
\end{tabular}
\caption{Fit of strength data using the GEL-S distribution}
\label{strengthGELS}
\end{table}

The studied strength data are very popular.
Asgharzadeh et al. \cite{AsgharzadehEsmaeiliNadarajahShih2013} represented the studied data using the general skew logistic (GSL) distribution introduced by these authors themselves and also using the following known models: the logistic (L), type III generalized logistic (GL3), skew logistic (SL), Azzalini's skew normal \cite{Azzalini1985} (SN), and Azzalini and Capitanio's skew $t$ \cite{AzzaliniCapitanio2003} (ST) distributions.
Kundu and Gupta \cite{KunduGupta2006} fitted the data to the Weibull distribution and provided its log-likelihood.
Mead et al. \cite{MeadAfifyHamedaniGhosh2017} also fitted to the data the distributions mentioned in the previous sub-section.
Ghitany et al. \cite{GhitanyAlJarallahBalakrishnan2013} examined these data using the exponentiated exponential (EE), exponentiated Rayleigh (ER) and exponentiated Pareto (EP) distributions.
More recently, Afify et al. \cite{AfifyNofalEbraheim2015} modeled these data using the Exponentiated
transmuted generalized Rayleigh (ETGR) distribution introduced by them, and the transmuted generalized Rayleigh (TGR) \cite{Merovci2014}, generalized Rayleigh (GR) \cite{SurlesPadgett2001}, and Rayleigh (R) distributions; and Afify et al. \cite{AfifyNofalYousofGebalyButt2015} did it using the transmuted Weibull Lomax (TWL) distribution introduced by them, and the Weibull Lomax (WL) \cite{TahirCordeiroMansoorZubair2015}, modified beta Weibull (MBW) \cite{Khan2015}, transmuted
complementary Weibull geometric (TCWG) \cite{AfifyNofalButt2014} 
and Lomax (Lx) distributions.
All these authors reported log-likelihoods, AIC values and/or SIC values.
The Tab. \ref{strengthGELS} collects both AIC and SIC values.

The AIC and SIC values associated to the GEL-S model based on the parameters indicated in Tab. \ref{strengthGELS} are also included in Tab. \ref{strengthTab}.
Since lower AIC and SIC values are preferred to choose a model, the highlighted ones in Tab. \ref{strengthTab}, according to these criteria the GEL-S model is favorite overall.

\begin{table}[!ht]
\centering
\begin{tabular}{lccc}
\hline
\multicolumn{1}{c}{Model} & $n_p$ $^{(*)}$ & AIC & SIC \\
\hline
BExFr & 5 & 122.70 & 129.71 \\
BFr & 4 & 121.06 & 126.66 \\
EE & 2 & 117.03 & 121.31 \\
EFr & 3 & 118.70 & 122.90 \\
EP & 2 & 120.09 & 122.89 \\
ER & 2 & 117.06 & 121.34 \\
ETGR & 4 & 122.97 & 131.50 \\
Fr & 2 & 121.80 & 124.60 \\
GEL-S & 2 & $\mathbf{116.41}$ & $\mathbf{120.70}$ \\
GIG & 5 & 123.06 & 130.07 \\
GL & 3 & 118.92 & 123.12 \\
GL3 & 3 & 123.59 & 130.02 \\
GR & 2 & 126.62 & 130.91 \\
GSL & 4 & 119.80 & 128.37 \\
\hline
\multicolumn{4}{l}{$^{(*)}$ $n_p$: number of parameters}
\end{tabular}
\hspace{2mm}
\begin{tabular}{lccc}
\hline
\multicolumn{1}{c}{Model} & $n_p$ $^{(*)}$ & AIC & SIC \\
\hline
L & 2 & 122.66 & 126.94 \\
Lx & 2 & 270.92 & 275.20 \\
McL & 5 & 123.01 & 130.02 \\
MBW & 5 & 135.91 & 146.62 \\
R & 1 & 189.04 & 191.18 \\
SL & 3 & 119.58 & 126.01 \\
SN & 3 & 117.80 & 124.23 \\
ST & 4 & 119.80 & 128.37 \\
TCWG & 4 & 134.89 & 143.46 \\
TGR & 3 & 124.63 & 131.06 \\
TWL & 5 & 129.68 & 140.39 \\
W & 2 & 124.30 & 128.59 \\
WL & 4 & 129.78 & 138.35 \\
ZBLL & 3 & 146.08 & 150.28 \\
\hline
\multicolumn{4}{l}{$^{(*)}$ $n_p$: number of parameters}
\end{tabular}
\caption{Strength data: AIC and SIC values}
\label{strengthTab}
\end{table}

Fig. \ref{strengthFood} presents a histogram with 25 bins of the strength data used and the overlay of the fitted GEL-S model reached.

\begin{figure}[!ht]
\centering \includegraphics[scale=0.40]{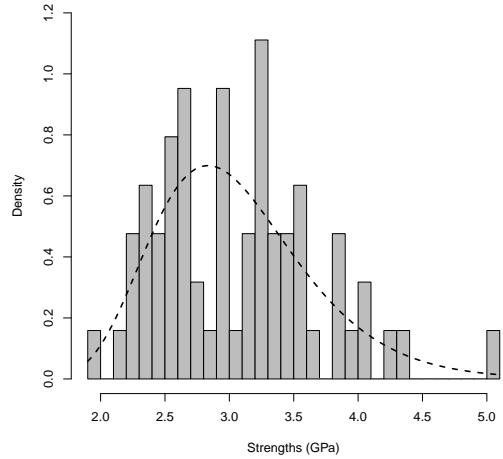} 
\caption{Strength data: Histogram and fitted GEL-S model (dashed line)}
\label{strengthFood} 
\end{figure}

\subsection{Failure Time Data}

In the third application, ball bearing failure times reported by Lieblein and Zelen \cite{LiebleinZelen1956} are taken into account.
These authors presented the results of manufacturers' tests of endurance of 213 batches of ball bearings, giving nearly 5000 deep-groove ball bearings.
Nevertheless these authors modeled each one of these batches, in the literature one of these batches has become ``one of the most frequently used data sets in literature for illustrating the applications of lifetime distributions'' in words of Yang \cite{Yang2001}.
This data set that contains 25 observations corresponds to the one described in the page 286 in \cite{LiebleinZelen1956}.
As Caroni remarks \cite{Caroni2002}, the original data are censored but this data feature has been forgotten, considering currently, then, that such data are uncensored.
Another fact on this data set is that commonly only 23 of its observations are used.

In this application, 23 failure times typically used in practice are examined, and they are assumed uncensored.
Notice that these data include 4 censored observations. These times are in millions of revolutions.
The MLEs for the parameters of the GEL-S distribution for these data and the corresponding reached log-likelihood are presented in Tab. \ref{ballGELS}.
Standard errors in brackets.

\begin{table}[!ht]
\centering
\begin{tabular}{cccc}
\hline
$k$ & $\hat{\alpha}^2$ & $\hat{\gamma}$ & $-l$ \\
\hline
27 (fixed) & $2.7918_{(0.3591)}$ & $0.4063_{(0.0072)}$ & 112.99 \\
\hline
\end{tabular}
\caption{Fit of failure time data using the GEL-S distribution}
\label{ballGELS}
\end{table}

As mentioned above, Lieblein and Zelen's data are very popular.
Applications concerning their modeling using distributions follows.
Gupta and Kundu \cite{GuptaKundu1999} examined those data using the gamma, Weibull and generalized exponential (GE) distributions.
These authors reported their log-likelihoods.
Kundu and Gupta \cite{KunduGupta2006} fitted the data to the Weibull distribution and provided its log-likelihood.
Mead et al. \cite{MeadAfifyHamedaniGhosh2017} also fitted the distributions mentioned in the previous sub-section and gave its log-likelihood associated to each model.
Ghitany et al. \cite{GhitanyAlJarallahBalakrishnan2013} examined these data using the exponentiated exponential (EE), exponentiated Rayleigh (ER) and exponentiated Pareto (EP) distributions, and presented their log-likelihoods.
Lawless \cite{Lawless} gave the fit of data using the log-normal distribution.
For all these distributions their AIC and SIC values are showed in Tab. \ref{ballGELS}.
In this table the corresponding values for the GEL-S model based on the parameters indicated in Tab. \ref{ballGELS} are also included.

In Tab. \ref{ballGELS}, lower AIC and SIC values preferred to choose a model are highlighted.
These outputs thus show that according to the AIC and the SIC the GEL-S model is favorite overall.

\begin{table}[!ht]
\centering
\begin{tabular}{lccc}
\hline
\multicolumn{1}{c}{Model} & $n_p$ $^{(*)}$ & AIC & SIC \\
\hline
GEL-S & 2 & $\mathbf{229.98}$ & $\mathbf{232.25}$ \\
Gamma & 3 & 231.70 & 235.10 \\
Weibull & 3 & 231.95 & 235.35 \\
GE & 3 & 231.53 & 234.93 \\
Log-normal & 2 & 230.20 & 232.47 \\
\hline
\multicolumn{4}{l}{$^{(*)}$ $n_p$: number of parameters}
\end{tabular}
\caption{Failure time data: AIC and SIC values}
\label{ballTab}
\end{table}

Fig. \ref{ballFood} presents a histogram with 10 bins of the failure time data used and the overlay of the fitted GEL-S model reached.

\begin{figure}[!ht]
\centering \includegraphics[scale=0.40]{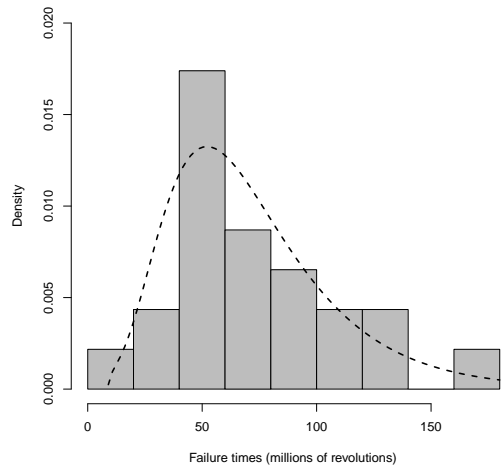} 
\caption{Failure time data: Histogram and fitted GEL-S model (dashed line)}
\label{ballFood} 
\end{figure}

\section{Discussion and Conclusion}
\label{sec6}

In this paper, a new right-skewed three-parameter distribution, with support $(\alpha,\infty)$ for some $\alpha\geq0$ and with probability density function showing exponential decays at its both tails, was introduced.
We called this distribution the generalized exponential log-squared (GEL-S) distribution.
The proposed original distribution had the limitation that one of its parameters, $k$, was not easily tractable as a continuous parameter, but it was tractable when it took non-negative integers.
This led to a reformulation of the proposed distribution by limiting $k$ to take non-negative integers.
We showed that the GEL-S distribution is close to well-known distributions such as the two-parameter and three-parameter log-normal and gamma distributions, but the new one does not generalizes any of these distributions.
Statistical properties of the GEL-S distribution were analyzed.
Closed forms for the $n$th moment and for statistics as the mean, variance, skewness, and kurtosis were provided.
Also, the mode and quantile function were studied.
The maximum likelihood method (MLE) for estimating the parameters of the distribution GEL-S was proposed, but it cannot be applied using derivatives since one of its parameters is not continuous.
This led to 
the formulation of a strategy that still applies derivatives.
The resulting procedure consisted in fixing k and thereafter, computing derivatives with respect to the other parameters.
Simulations conducted to assess the performance of the above-strategy for estimating parameters were performed, finding that for small samples the true parameters could not be recovered.
Despite this last issue, applications performed on three well-known real light-tailed and right-skewed data sets related to different domains showed that the new distribution is a reasonable alternative to other natural competitors.
Thus, the new distribution seems to be a promising model for representing light-tailed and right-skewed data.

We also give examples in which the arctangent distribution is a reasonable alternative to other common lifetime distributions

\appendix

\section{Proofs}
\label{Proofs}


\begin{proof}[Deduction of $C$ given in (\ref{deff})]
Noting that
\begin{eqnarray*}
\int_\alpha^\infty x^k\,e^{-\frac{1}{2\gamma^2}\left(\log(x-\alpha)\right)^2}dx
 & = & \int_0^\infty (y+\alpha)^k\,e^{-\frac{1}{2\gamma^2}\left(\log y\right)^2}dy,\quad y=x-\alpha \\
 & = & \int_0^\infty \sum_{i=0}^k{{k}\choose{i}}y^i\alpha^{k-i}\,e^{-\frac{1}{2\gamma^2}\left(\log y\right)^2}dy \\
 & = & \sum_{i=0}^k{{k}\choose{i}}\alpha^{k-i}\int_{-\infty}^\infty e^{-\frac{1}{2\gamma^2}z^2+(i+1)z}dz,\quad z=\log y \\
 & = & \sum_{i=0}^k{{k}\choose{i}}\alpha^{k-i}e^{\frac{1}{2}\gamma^2(i+1)^2}\int_{-\infty}^\infty e^{-\frac{1}{2}\left(\frac{z}{\gamma}-\gamma(i+1)\right)^2}dz \\
 & = & \sqrt{2\pi}\gamma\sum_{i=0}^k{{k}\choose{i}}\alpha^{k-i}e^{\frac{1}{2}\gamma^2(i+1)^2},\quad u=\frac{z}{\gamma}-\gamma(i+1),
\end{eqnarray*}
it follows
$$
1=\int_\alpha^\infty Cx^k\,e^{-\frac{1}{2\gamma^2}\left(\log(x-\alpha)\right)^2}dx=C\sqrt{2\pi}\gamma\sum_{i=0}^k{{k}\choose{i}}\alpha^{k-i}e^{\frac{1}{2}\gamma^2(i+1)^2},
$$
and $C$ is then deduced.
\end{proof}


\begin{proof}[Deduction of $F$ given in (\ref{defF})]
Noting that, for $x>\alpha$,
\begin{eqnarray*}
\int_\alpha^x w^k\,e^{-\frac{1}{2\gamma^2}\left(\log(w-\alpha)\right)^2}dw
 & = & \int_0^{x-\alpha} (y+\alpha)^k\,e^{-\frac{1}{2\gamma^2}\left(\log y\right)^2}dy,\quad y=w-\alpha \\
 & = & \int_0^{x-\alpha} \sum_{i=0}^k{{k}\choose{i}}y^i\alpha^{k-i}\,e^{-\frac{1}{2\gamma^2}\left(\log y\right)^2}dy \\
 & = & \sum_{i=0}^k{{k}\choose{i}}\alpha^{k-i}\int_{-\infty}^{\log({x-\alpha})} e^{-\frac{1}{2\gamma^2}z^2+(i+1)z}dz,\quad z=\log y \\
 & = & \sum_{i=0}^k{{k}\choose{i}}\alpha^{k-i}e^{\frac{1}{2}\gamma^2(i+1)^2}\int_{-\infty}^{\log({x-\alpha})} e^{-\frac{1}{2}\left(\frac{z}{\gamma}-\gamma(i+1)\right)^2}dz \\
 & = & \sqrt{2\pi}\gamma\sum_{i=0}^k{{k}\choose{i}}\alpha^{k-i}e^{\frac{1}{2}\gamma^2(i+1)^2}\Phi\left(\frac{\log({x-\alpha})}{\gamma}-\gamma(i+1)\right), \\
 & & \qquad u=\frac{z}{\gamma}-\gamma(i+1).
\end{eqnarray*}
we have that, for $x>\alpha$,
$$
F(x)=\int_\alpha^x f(w)dw
=C\sqrt{2\pi}\gamma\sum_{i=0}^k{{k}\choose{i}}\alpha^{k-i}e^{\frac{1}{2}\gamma^2(i+1)^2}\Phi\left(\frac{\log({x-\alpha})}{\gamma}-\gamma(i+1)\right),
$$
and the deduction of $F$ follows.
\end{proof}


\begin{proof}[Deduction of the $n$th moment in (\ref{momn})]
Let $n=0,1,2,\ldots$.
Noting that
$$
E\big[X^n\big] = C\,\int_a^\infty x^{n+k}\,e^{-\frac{1}{2\gamma^2}\left(\log(x-\alpha)\right)^2}dx,
$$
then following a procedure as the one applied to deduce $C$ given in (\ref{deff}) but considering $n+k$ instead of $k$ gives
$$
E\big[X^n\big] = \sqrt{2\pi}\gamma C\,\sum_{i=0}^{n+k}{{n+k}\choose{i}}\alpha^{n+k-i}e^{\frac{1}{2}\gamma^2(i+1)^2}.
$$
\end{proof}


\begin{proof}[Proof of Proposition \ref{PropMode}]
The result follows by derivating $f$ given by (\ref{deff}) and solving the quation $f'(x)=0$, i.e.
$$
C\left(kx^{k-1}e^{-(2\gamma^2)^{-1}\left(\log(x-\alpha)\right)^2}-\frac{1}{\gamma^2}x^{k}\frac{\log(x-\alpha)}{x-\alpha} e^{-(2\gamma^2)^{-1}\left(\log(x-\alpha)\right)^2}\right)=0,
$$
which implies, by $x>\alpha\geq0$,
\begin{equation}\label{eqMode}
x\,\log(x-\alpha)=k\gamma^2(x-\alpha).
\end{equation}
If $\alpha=0$, then $x=0$ is a solution of (\ref{eqMode}). Since the left-side of this equation is a convex function with derivative $\log x+1$, and the right-side of the equation is a right line with slope $k\gamma^2$, there is a second non-negative solution.
In this case $x=0$ is not a mode since
$$
\lim_{x\to0^+}f(x)=0,
$$
so the only one positive solution is the mode.

Assuming $\alpha>0$, the left-side of (\ref{eqMode}) has derivative, for $x>\alpha$,
$$
\log(x-\alpha)+\frac{x}{x-\alpha},
$$
implying that $x\,\log(x-\alpha)$ is increasing and satisfying
$$
\lim_{x\to\alpha^+}x\,\log(x-\alpha)=-\infty,\quad
\lim_{x\to\infty}x\,\log(x-\alpha)=\infty,
$$
whereas right-side of (\ref{eqMode}) has derivative $k\gamma^2$, implying that $k\gamma^2(x-\alpha)$ and satisfying
$$
\lim_{x\to\alpha^+}k\gamma^2(x-\alpha)=0.
$$
Hence (\ref{eqMode}) has an only one positive solution that corresponds to the mode of $X$.

The claims of the proposition then follow.
\end{proof}





\end{document}